\documentclass[reqno,a4paper]{amsart}

\usepackage{mathrsfs,amsmath,amssymb,graphicx}
\usepackage{mathtools}%for smashoperator
\usepackage[unicode]{hyperref}
\usepackage{layout}
\usepackage[english]{babel} 
\usepackage{color}
\usepackage{enumitem} 
\usepackage{float}
\usepackage{subcaption}
\usepackage{graphicx}
\usepackage[percent]{overpic} 
 
\usepackage{tikz}
\usetikzlibrary{calc}
\usetikzlibrary{matrix,arrows,decorations,chains,positioning}
\makeatletter
\@namedef{subjclassname@2020}{\textup{2020} Mathematics Subject Classification}
\makeatother

\theoremstyle{plain}
\newtheorem*{theorem*}{Theorem} 
\newtheorem{theorem}{Theorem}[section]
\newtheorem{lemma}[theorem]{Lemma}

\newtheorem{corollary}[theorem]{Corollary}

\theoremstyle{definition}

\newtheorem{question}[theorem]{Question} 
\newtheorem{problem}[theorem]{Problem}

\newcommand{\mc}[1]{{\mathcal #1}}

\newcommand{\mirrormc}[1]{\overline{\mathcal #1}}
\newcommand{\mirror}[1]{\overline{#1}}
 
\newcommand{\para}{(l,m,n,p)}
\newcommand{\paraprime}{(l',m',n',p')}
\newcommand{\lpara}{(\ast,m,n,p)} 

\newcommand{\rl}{{R/L}}
\newcommand{\lr}{{L/R}} 
\newcommand{\Vr}{\mathcal{V}_R}
\newcommand{\Vl}{\mathcal{V}_L}
\newcommand{\Vrl}{\mathcal{V}_\rl}

\newcommand{\hslopeA}{r_h^{\mathcal{A}}}
\newcommand{\hslopeB}{r_h^{\mathcal{B}}}
\newcommand{\vslopeA}{r_v^{\mathcal{A}}}
\newcommand{\vslopeB}{r_v^{\mathcal{B}}}
\newcommand{\hdetA}{d_h^{\mathcal{A}}}
\newcommand{\hdetB}{d_h^{\mathcal{B}}}
\newcommand{\vdetA}{d_v^{\mathcal{A}}}
\newcommand{\vdetB}{d_v^{\mathcal{B}}}
\newcommand{\MP}{\mathbf{\Phi}}
\newcommand{\MN}{\mathbf{\Lambda}}
\newcommand{\minus}{{\raisebox{.12em}{\scalebox{.57}[.57]{$-$}}}}
\newcommand{\plus}{{\raisebox{.12em}{\scalebox{.57}[.57]{$+$}}}}
\newcommand{\plusminus}{{\raisebox{.12em}{\scalebox{.57}[.57]{$\pm$}}}}
\newcommand{\minusplus}{{\raisebox{.12em}{\scalebox{.57}[.57]{$\mp$}}}}
\newcommand{\fivetwo}{\mathbf{5}_2}
\newcommand{\sixthirdteen}{\mathbf{6}_{13}}

\newcommand{\sphere} {S^3}

\newcommand{\HK}{V}
\newcommand{\pair}{(\sphere,\HK)}

\newcommand{\homeo}[1]{{\rm Homeo}(#1)} 
\newcommand{\phomeo}[1]{{\rm Homeo_\plus}(#1)´} 
 
\newcommand{\openrnbhd}[1]{\mathring{\mathfrak N}(#1)}
\newcommand{\Compl}[1]{E(#1)}
\newcommand{\front}{\partial_f}

\newcommand{\mcg}[1]{\mathrm{MCG}(#1)}
\newcommand{\pmcg}[1]{\mathrm{MCG}_+(#1)}
\newcommand{\sym}{\mathrm{MCG}}
\newcommand{\psym}{\mathrm{MCG}_+}

\newcommand{\Ztwo}{\mathbb{Z}_2}
\newcommand{\image}[1]{{\mathrm Im}(#1)}

\newcommand{\lqu}{``} 
\newcommand{\cout}[1]   {}%comment out

\numberwithin{equation}{section}

\title[Symmetry]{On symmetry and exterior problems of knotted handlebodies}

\author{Yuya Koda, Makoto Ozawa, Yi-Sheng Wang}
\address{Department of Mathematics, Hiyoshi Campus, Keio University, 4-1-1, Hiyoshi, Kohoku, Yokohama, 223-8521, Japan~ \slash ~ 
International Institute for Sustainability with Knotted Chiral Meta Matter (WPI-SKCM$^2$), Hiroshima University, 1-3-1 Kagamiyama, Higashi-Hiroshima, 739-8526, Japan}
\email{koda@keio.jp}
\address{Department of Natural Sciences, Faculty of Arts and Sciences, Komazawa University, 1-23-1 Komazawa, Setagaya-ku, Tokyo, 154-8525, Japan}
\email{w3c@komazawa-u.ac.jp}
\address{National Sun Yat-sen University, Kaohsiung 804, Taiwan}
\email{yisheng@math.nsysu.edu.tw}

\date{\today}

\begin{document}
 
\subjclass[2020]{Primary 57K30; Secondary 57M12, 57K10, 57K12}
\keywords{}
\thanks{Y. K. is supported by JSPS KAKENHI Grant Numbers JP20K03588, JP21H00978,  JP23H05437 and JP24K06744.
M. O. is partially supported by Grant-in-Aid for Scientific Research (C) (No. 17K05262), The Ministry of Education, Culture, Sports, Science and Technology, Japan. 
Y-S. W. is supported by National Sun Yat-sen University and MoST (grant no. 110-2115-M-001-004-MY3), Taiwan.}

\begin{abstract}
The paper concerns two classical problems in knot theory pertaining to knot symmetry and knot exteriors. In the context of a knotted handlebody $V$ in a $3$-sphere $\sphere$, the symmetry problem seeks to classify the mapping class group of the pair $\pair$, whereas the exterior problem examines to what extent the exterior $\Compl V$ determines or fails to determine the isotopy type of $V$. The paper determines the symmetries of knotted genus two handlebodies arising from hyperbolic knots with non-integral toroidal Dehn surgeries, and solve the knot exterior problem for them. A new interpretation and generalization of a Lee-Lee family of knotted handlebodies is provided.
\end{abstract}
\maketitle
%%%%%%%%%%%%%%%%%%%%%%%%%%%%%%%%%%%%%%%%%%%%%%%%%%%%%%%%%%%%%%%%%%%%%%%%%%%%%
 
\section{Introduction}\label{sec:intro}
We work in the piecewise linear category. A genus $g$ handlebody-knot $\pair$ is a genus $g$ handlebody $V$ embedded in the $3$-sphere $\sphere$. To measure the symmetry of $\pair$, one considers the mapping class group (resp.\ positive symmetry group):   
\[\sym\pair:=\pi_0\left(\homeo {\sphere,V}\right)\
%( \  \text{resp.}\  \sym\pair:=\pi_0\left(\phomeo {\sphere,V} \right) \ )  ,
\mbox{ (resp. $ \sym\pair:=\pi_0\left(\phomeo {\sphere,V} \right)$), }\\
 \]
where $\homeo{\sphere,V}$ (resp.\ $\phomeo{\sphere,V}$)
is the group of homeomorphisms (resp.\ orientation-preserving homeomorphisms). 
In the case $\pair$ is trivial---namely, the exterior $\Compl\HK:=\overline{\sphere-\HK}$ is a handlebody, $\sym\pair$ is the \emph{Goeritz group} \cite{Goe:33}. In general, we call $\sym\pair$ the \emph{symmetry group} of $\pair$. A natural yet difficult question asks 
%The following is a classical and remains unsolved problem. 
\begin{question}
Whether $\sym\pair$ is finitely generated.  
\end{question}
%When the genus $g>2$, the finite generation of $\sym\pair$ 
 %largely unknown---notably, the Powell conjecture remains unsolved for $g>3$ \cite{Pow:80} (see \cite{FreSch:18} and \cite{ChoKodLee:24}). 
When the genus $g$ is greater than $2$, the finite generation of $\sym\pair$ is largely unknown. 
A notable example is the \emph{Powell conjecture}, which asserts that the Goeritz group of the genus-$g$ Heegaard splitting of $S^3$—equivalently, the symmetry group of the trivial handlebody-knot $\pair$—is finitely generated, offering a specific generating set of four elements. 
The conjecture has been verified for $g = 3$ (see \cite{FreSch:18,ChoKodLee:24} by Freedman, Scharlemann, and Cho, Lee, the first-named author), whereas it remains open for $g \geq 4$.
When $g=1$, the group $\sym\pair$, equivalent to the symmetry group of a knot, is known to be finitely generated, and furthermore, it is finite if and only if $\sym\pair$ is either cyclic or dihedral, if and only if the core of $\pair$ is a hyperbolic or torus knot, or a cable of a torus knot; see Kawauchi \cite{Kaw:96} and Sakuma \cite{Sak:89}. 
When $g=2$, by works in Scharlemann \cite{Sch:04}, Akbas \cite{Akb:08}, Cho \cite{Cho:08} and the first-named author \cite{Kod:15}, it is now known the symmetry group $\sym\pair$ is finitely generated if $\Compl\HK$ is reducible. In addition, the result by Funayoshi and the first-named author \cite{FunKod:20} implies that $\sym\pair$ is finite if and only if $\pair$ is non-trivial and atoroidal---namely, $\Compl\HK$ containing no incompressible torus. Not much is known, however, regarding the finite group structure. 
\begin{problem}\label{prob:structure}
Classify the group structure of $\sym\pair$ when it is finite. 
\end{problem}
The case $g>2$ appears to be an uncharted territory. Efforts have been made to understand the finite symmetry group structure when $g=2$; see, for instance, \cite{Kod:15} by the first-named author, and \cite{Wan:21,Wan:23,Wan:24ii} by the third-named author, yet a complete classification of the finite group structure remains awaited. 
 
Another classical problem in knot theory, the 
\emph{knot complement problem}, asks whether the knot exterior determines the knot. By Gordon-Luecke \cite{GorLue:89}, it is the case if $g=1$, while it fails in general if $g>1$; see Suzuki \cite{Suz:75} for $g>2$, and Motto \cite{Mott:90}, Lee-Lee \cite{LeeLee:12}, as well as \cite{BelPaoWan:20a} by Bellettini, Paolini and the third-named author for $g=2$. A question thus arises as to 
\begin{question}\label{ques:exterior}
What information gets lost in passing from $\pair$ to $\Compl V$, and how to recover it?  
\end{question} 
The current work considers Problem \ref{prob:structure} and Question \ref{ques:exterior} in the case of genus two handlebody-knots, abbreviated to \emph{handlebody-knots} hereinafter. 
\subsection{Hyperbolization and Classification of annuli}
As with the case of knots, Thurston's hyperbolization \cite{Thu:82} classifies non-trivial handlebody-knots into three classes based on annuli and tori in their exteriors. More specifically, a handlebody-knot $\pair$ is \emph{cylindrical} or \emph{toroidal} if its exterior $\Compl V$ admits an essential annulus or torus, respectively. By Tsukui \cite{Tsu:75}, a handlebody-knot exterior admits an essential disk if it is trivial or toroidal. Thus by the hyperbolization theorem, a non-trivial handlebody-knot belongs to exactly one of the three classes: toroidal, atoroidal-cylindrical, or hyperbolic handlebody-knots. This should be contrasted with the three classes of knots: satellite, torus, and hyperbolic knots, respectively.

\begin{figure}[b]
\begin{subfigure}[t]{.32\linewidth}
		\centering
		\begin{overpic}[scale=.13,percent]{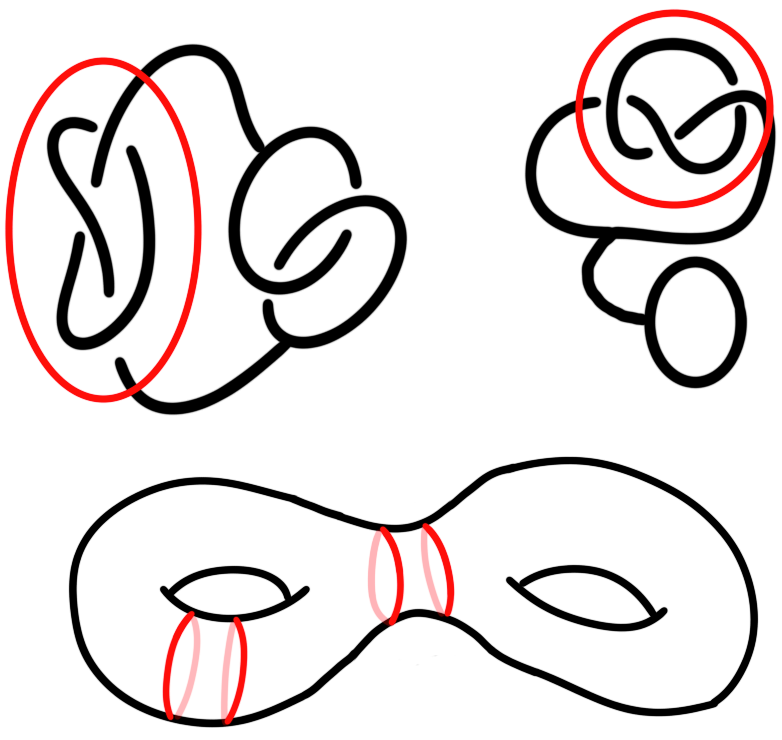}
			\put(15,9){\footnotesize $l_1$}
			\put(32,8){\footnotesize $l_2$}
			\put(49.5,8.5){\footnotesize $l_1$}
			\put(55.5,8.5){\footnotesize $l_2$} 
		\end{overpic}
		\caption{Type $1$.}
		\label{fig:typeone}
\end{subfigure} 
\begin{subfigure}[t]{.32\linewidth}
		\centering
		\begin{overpic}[scale=.13,percent]{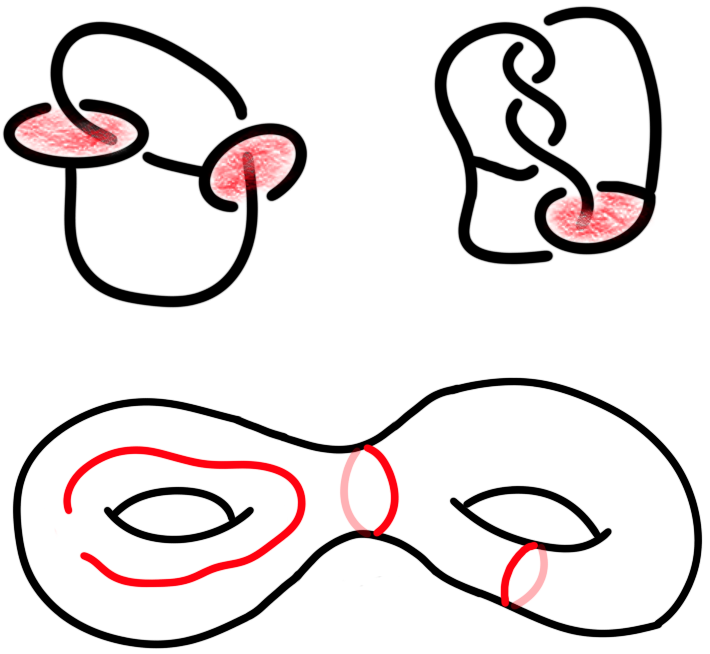}
		\put(10,15){\footnotesize $l_1$}
        \put(51,8.5){\footnotesize $l_2$}
		\put(65.5,1){\footnotesize $l_2$}
		\end{overpic}
		\caption{Type $2$.}
		\label{fig:typetwo}
\end{subfigure}
\begin{subfigure}[t]{.34\linewidth}
		\centering
		\begin{overpic}[scale=.13,percent]{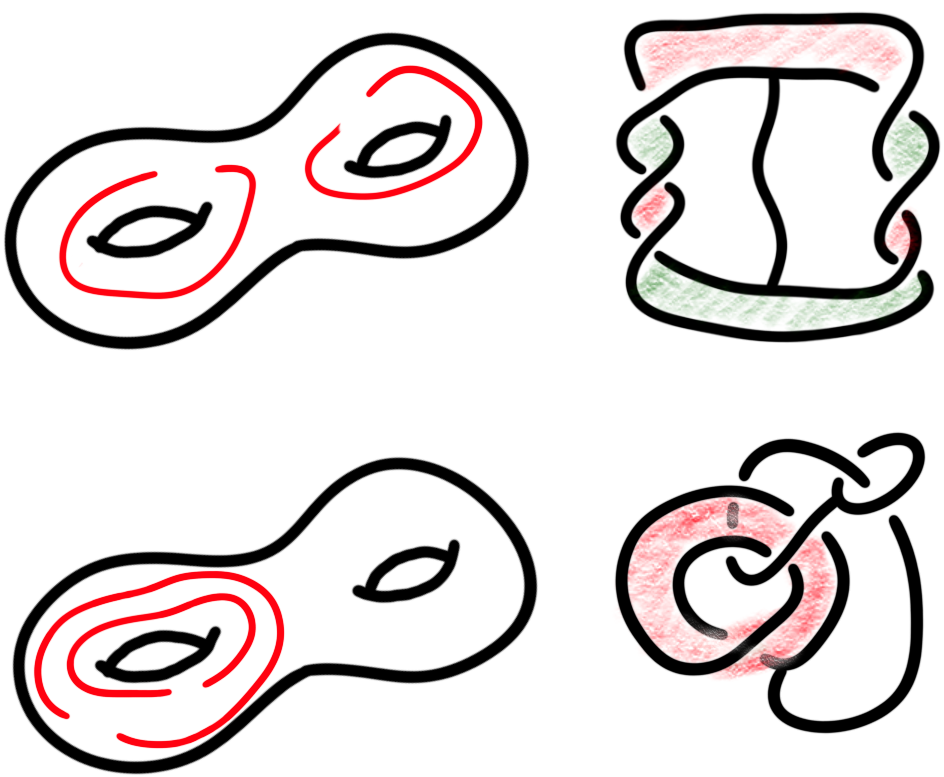}
			\put(17,62){\footnotesize $l_1$}
			\put(36,69){\footnotesize $l_2$}  
			\put(17.7,6){\footnotesize $l_1$}
			\put(7,4){\footnotesize $l_2$}  
		\end{overpic}
		\caption{Type $3$-$3$; Type $3$-$2$.}
		\label{fig:typethree}
\end{subfigure}
\caption{Types of the annulus $A$, and its boundary $\partial A=l_1\cup l_2$.}
\end{figure}
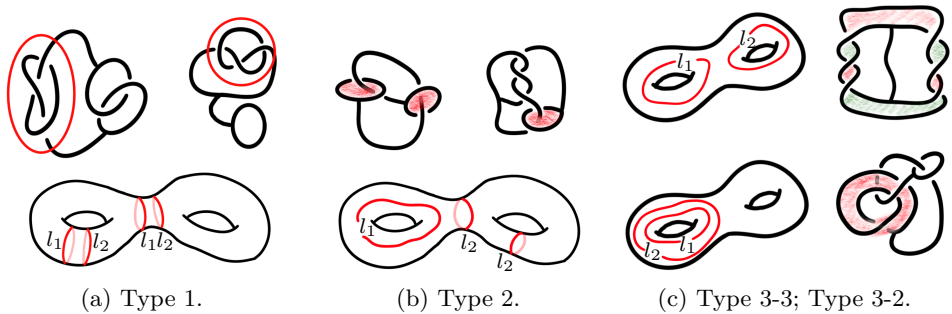

While essential annuli in a knot exterior are always cabling, essential annuli may occur in a handlebody-knot exterior in a variety of ways. Based on the boundary of essential annuli in relation to the handlebody, the first and second-named authors \cite{KodOzaGor:15} classify them into four types: an essential annulus $A$ is of \emph{type $1$} if both components of $\partial A$ bound disks in $\HK$; see Fig.\ \ref{fig:typeone}, and is of \emph{type $2$} if exactly one component of $\partial A$ bounds a disk in $\HK$; see Fig.\ \ref{fig:typetwo}. Suppose no component of $\partial A$ bounds a disk in $\HK$, then $A$ is of \emph{type $3$} if there is a compressing disk of $\partial\HK\subset\sphere$
disjoint from $\partial A$ (Fig.\ \ref{fig:typethree}), and it is of \emph{type $4$} otherwise. Type $4$ annuli are the hardest to visualize, yet they have strong connection with non-integral toroidal Dehn surgery on hyperbolic knots. Note that 
%$\HK$ being of genus two implies 
the boundary components of a type $1$ annulus are parallel in $\partial\HK$, and hence it induces an incompressible torus in $\Compl\HK$ and implies toroidality; particularly, $\sym\pair$ is not finite in this case. 

On the other hand, if $\pair$ is atoroidal, then one can further divide a type $3$ annuli into two families: if boundary components of a type $3$ annulus $A$ are parallel in $\partial\HK$, then $A$ is of \emph{type $3$-$2$}; see the bottom figure in Fig.\ \ref{fig:typethree}; otherwise it is of \emph{type $3$-$3$}; see the top figure in Fig.\ \ref{fig:typethree}. A type $3$-$2$ annulus $A$ cuts off a solid torus $U$ from $\Compl\HK$, and the \emph{slope} of $A$ is defined to be the boundary slope of $A$ with respect to $U$. 
Similarly, given a type $3$-$3$ annulus $A$, by Lee-Lee \cite{LeeLee:12}, Funayoshi and the first-named author \cite{FunKod:20}, the third-named author \cite[Lemma $2.9$]{Wan:23}, there is a unique separating disk $D$ in $\HK$ disjoint from $\partial A=l_1\cup l_2$, which cuts $\HK$ 
into two solid tori $V_1,V_2$ with $l_i\subset \partial V_i$, $i=1,2$. The \emph{slope pair} $(r_1,r_2)$ is then defined as the slope $r_i$ of $l_i$ with respect to $V_i$.

%It turns out that a type $4$ annulus $A$ is strongly connected to non-integral toroidal Dehn surgery on a hyperbolic knot. 
To see the relation between a type $4$ annulus $A$ and non-integral toroidal Dehn surgery, we recall that, by  
the first and second-named authors \cite{KodOzaGor:15}, components of $\partial A$ are parallel in $\partial\HK$, and hence $A$ cuts off a solid torus $U$ from $\Compl\HK$.
%; the slope of $A$ is then defined to be the slope of its core with respect to $U$. 
Secondly, by the definition of a type $4$ annulus, $T:=\partial\HK-U$ is an incompressible twice-punctured torus in $\Compl U$, and cuts $\Compl U$ into two
$\partial$-reducible $3$-manifolds: one is $V$ and the other, denoted by $W$, is in $\Compl V$. We say $A$ is of \emph{type $4$-$1$} if $W$ is a handlebody, and is of \emph{type $4$-$2$} otherwise. In particular, $A$ is of type $4$-$2$ if and only if the frontier of the compression body of $\partial W$ is a torus or two tori. Also, $A$ is of type $4$-$1$ if and only if $\pair$ is atoroidal, if and only if the core of $U$ is a hyperbolic knot. In summary, the core of $U$ is a hyperbolic knot whose exterior admits a twice-punctured incompressible torus $T$ with non-integral slope. Conversely, if a hyperbolic knot $K$ admits a non-integral toroidal Dehn surgery, then its exterior $\Compl K$ admits a twice-punctured incompressible torus $T$ with non-integral slope by Gordon-Luecke \cite{GorLue:00}. Since $K$ is hyperbolic, $T$ cuts $\Compl K$ into two handlebodies $V,W$, and $T$ being incompressible with non-integral boundary slope of $T$ implies the annulus $A:=\Compl K\cap W\subset \Compl V$ (resp.\ $A:=\Compl K\cap V\subset \Compl W$) is essential and of type $4$-$1$ in the exterior of the handlebody-knot $\pair$ (resp.\ $(\sphere,W)$). 
The above relation is made concrete through the 
tangles in Eudave-Mu\~noz \cite{Eud:97}, \cite{Eud:02}.

\subsection{Eudave-Mu\~noz knots} 
Recall the construction from Eudave-Mu\~noz \cite{Eud:02}:
consider the knot $k_e$ in the $3$-sphere $\sphere$ and the $3$-ball $B_e$ that meets $k_e$ at two strings in Fig.\ \ref{fig:emtangle}, where $\mc A,\mc B,\mc C$ are rational tangles
$R(l), R(p,-2,m,l), \mc C=R(-n,2,m-1,2,0)$, respectively; see Fig.\ \ref{fig:even},\ref{fig:odd}, \ref{fig:vertical}, \ref{fig:horizontal} for the convention. 
Let $\pi:M\rightarrow \sphere$ be the double cover branched along $k_e$, and observe that $k_e \subset\sphere$ is a trivial knot if $n=0$ or $p=0$, so $M=\sphere$ in this case. In particular, the core $K\para$ of the preimage $\pi^{-1}(B_e)$ of $B_e$ is a knot in $\sphere$ when $np=0$. 

\begin{theorem}[{\cite{Eud:02}}]\label{intro:Eudave_thm}
The knot $K\para$ is hyperbolic %with $T$ a twice-punctured incompressible torus in its exterior---hence $K\para$ 
and admits a non-integral toroidal Dehn surgery 
if either
\begin{align}  
\bullet\quad p=0,\ & \text{and}\ l\neq 0, \pm 1, m\neq 0,\ (l,m)\neq (\pm 2,\pm 1),\ (m,n)\neq (1,0),(-1,1), \text{or}\nonumber\\  
%\end{align}
\bullet\quad%\begin{equation}\label{eq:nzero}
n=0,\ & \text{and}\ l\neq 0, \pm 1, m\neq 0, 1,\  (l,m,p)\neq (2,2,1), (-2,-1,0).  
\label{eq:constraints}
\end{align}
\end{theorem}
The knot $K\para$ is called an \emph{Eudave-Mu\~noz knot} if $\para$ satisfies \eqref{eq:constraints}. In this case, the preimage $T_e$ of the disk $D_e$ in Fig.\ \ref{fig:emtangle} gives us a twice-punctured incompressible torus $T_e$ in the exterior of $K\para$.

\begin{theorem}[{\cite{GorLue:04}}]
\label{intro:gorden_luecke_thm}
%A hyperbolic knot $K$ admits a non-integral, toroidal Dehn surgery if and only if $K$ admits an twice-punctured incompressible torus $T$ with non-integral boundary slope if and only if 
%$K$ is equivalent to an Eudave-Mu\~noz knot with $T$ isotopic to $T_e$.   
For a hyperbolic knot $K$, the following (i)--(iii) are equivalent. 
\begin{enumerate}[label=\textnormal{(\roman*)}]
\item
$K$ admits a non-integral, toroidal Dehn surgery. 
\item\label{itm:incomp_twice_punc}
$\Compl K$ admits %an 
a twice-punctured incompressible torus with non-integral boundary slope. 
\item
$K$ is equivalent to an Eudave-Mu\~noz knot. 
\end{enumerate}
In addition, every twice-punctured incompressible torus in $\Compl K$ in \ref{itm:incomp_twice_punc} is isotopic to $T_e$.
\end{theorem}
 
Hereinafter, we assume the parameters $l,m,n,p$ 
satisfy \eqref{eq:constraints}.  
Let $B_R,B_L\subset\Compl {B_e}$ be the $3$-balls cut off by $D_e$; see Fig.\ \ref{fig:emtangle}. Then the preimages $V_R,V_L$ of $B_R,B_L$ under $\pi$, respectively, are the two handlebodies cut off by $T_e$ from the exterior of $K\para$. For the sake of simplicity, we denote by $\Vrl\para$ the handlebody-knots $(\sphere,V_\rl)$, and call them the \emph{right/left handlebody-knots} induced by $K\para$. Since $\mathcal{A}=R(l)$ and $\mathcal{B}=R(p,-2,m,-l)$ in Fig.\ \ref{fig:emtangle}, the isotopy type of $\Vl\para$ is independent of $l$. When only the isotopy type of $\Vl\para$ is relevant, we use the notation $\Vl(\ast,m,n,p)$. As a corollary of Theorems \ref{intro:Eudave_thm} and \ref{intro:gorden_luecke_thm}, we have the following.
 
\begin{corollary}
$\pair$ admits a type $4$-$1$ annulus if and only if $\pair$ is equivalent to a right or left handlebody-knot induced by $K\para$, for some $\para$. 
\end{corollary} 
 
The simplest left and right handlebody-knots 
are, respectively, $\fivetwo$ and $\sixthirdteen$ in the knot table \cite{IshKisMorSuz:12} by Ishii-Kishimoto-Moriuchi-Suzuki.  
Up to mirror image, they are induced by $K(3,1,1,0)$; see Fig.\ \ref{fig:kthreeoneone}. In general though,  Eudave-Mu\~noz knots and their induced handlebody-knots are hard to visualize; efforts have been made to concretely describe them; see Eudave-Mu\~noz \cite{Eud:24}. 
On the other hand, in terms of JSJ decomposition, a handlebody-knot exterior that admits a type $4$-$1$ annulus is rather restrictive. 
 
\begin{figure}[t]
\begin{subfigure}[t]{.49\linewidth}
\centering
\begin{overpic}[scale=.14,percent]
{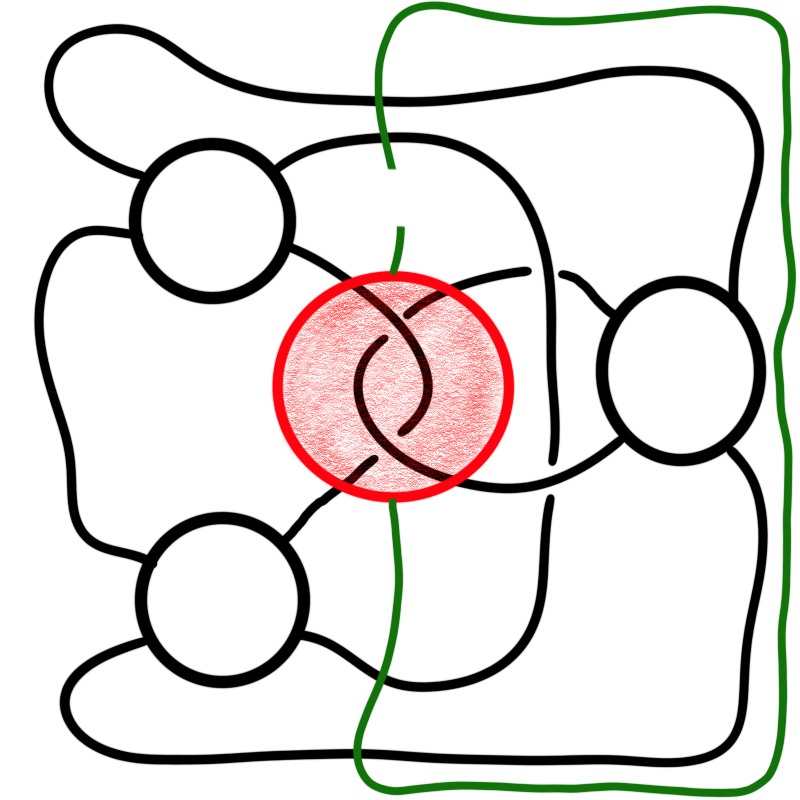}
\put(82,51){$\mathcal{C}$}
\put(21,70){\large $\mathcal{A}$}
\put(24,20){\large $\mathcal{B}$}
\put(45,72){\small $D_e$}
\put(36.5,44){\footnotesize $B_e$}
\put(77,7){$k_e$}
\put(50.5,92){\footnotesize $B_R$}
\put(34.5,92.5){\footnotesize $B_L$}
\end{overpic}
\caption{Eudave-Mu\~noz construction.}
\label{fig:emtangle}
\end{subfigure}
\begin{subfigure}[t]{.5\linewidth}
\centering
\begin{overpic}[scale=.15,percent]
{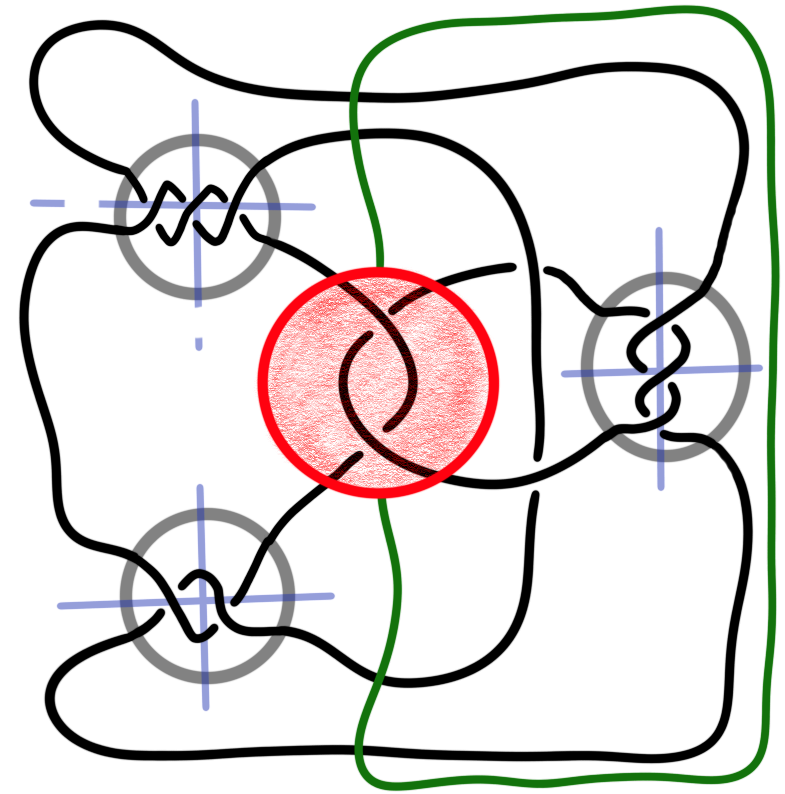}
\put(8 ,73){\small $h$}
\put(23 ,57.5){\small $v$}
\put(77,7){$k_e$}
\put(50.5,90.8){\tiny $B_R$}
\put(34.5,91.5){\tiny $B_L$}
\end{overpic}
\caption{$\para=(3,1,1,0)$.
}
\label{fig:kthreeoneone}
\end{subfigure}
%\caption{}
%\end{figure}
%
%
%\begin{figure}[t]
	\begin{subfigure}[t]{.47\linewidth}
		\centering
		\begin{overpic}[scale=.12,percent]{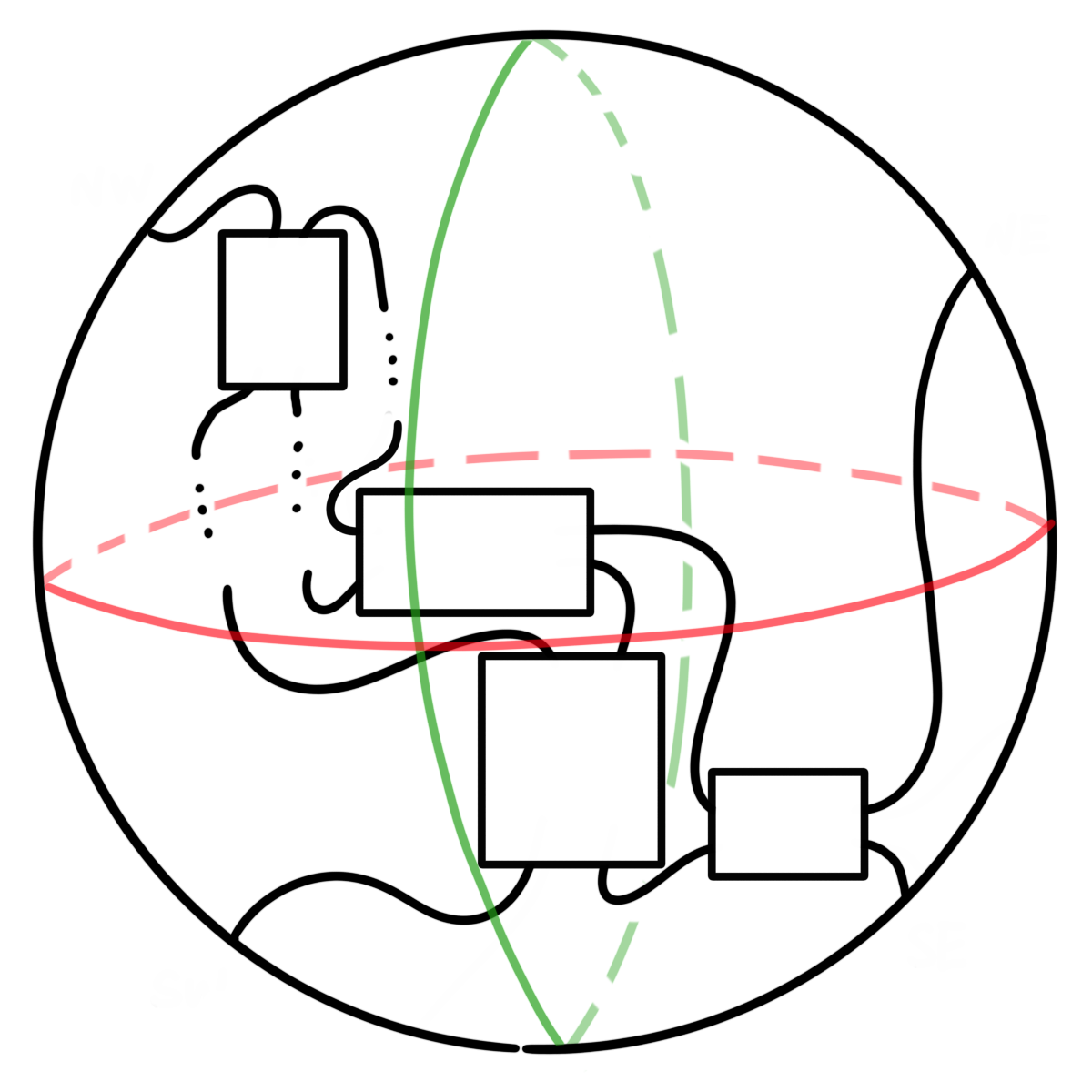}
			\put(70,22){$a_n$}
			\put(45,27){$a_{n-1}$}
			\put(37,47){$a_{n-2}$}
			\put(22,70){$a_1$}
			\put(6,80){$A$}
			\put(90,77){$B$}
			\put(84,12){$C$}			 
			\put(14,7){$D$}
            \put(90,42){\large $h$}			 
            \put(48,90){\large $v$}
		\end{overpic}
		\caption{$n$ is even.}
		\label{fig:even}
	\end{subfigure} 
	\begin{subfigure}[t]{.47\linewidth}
		\centering
		\begin{overpic}[scale=.12,percent]{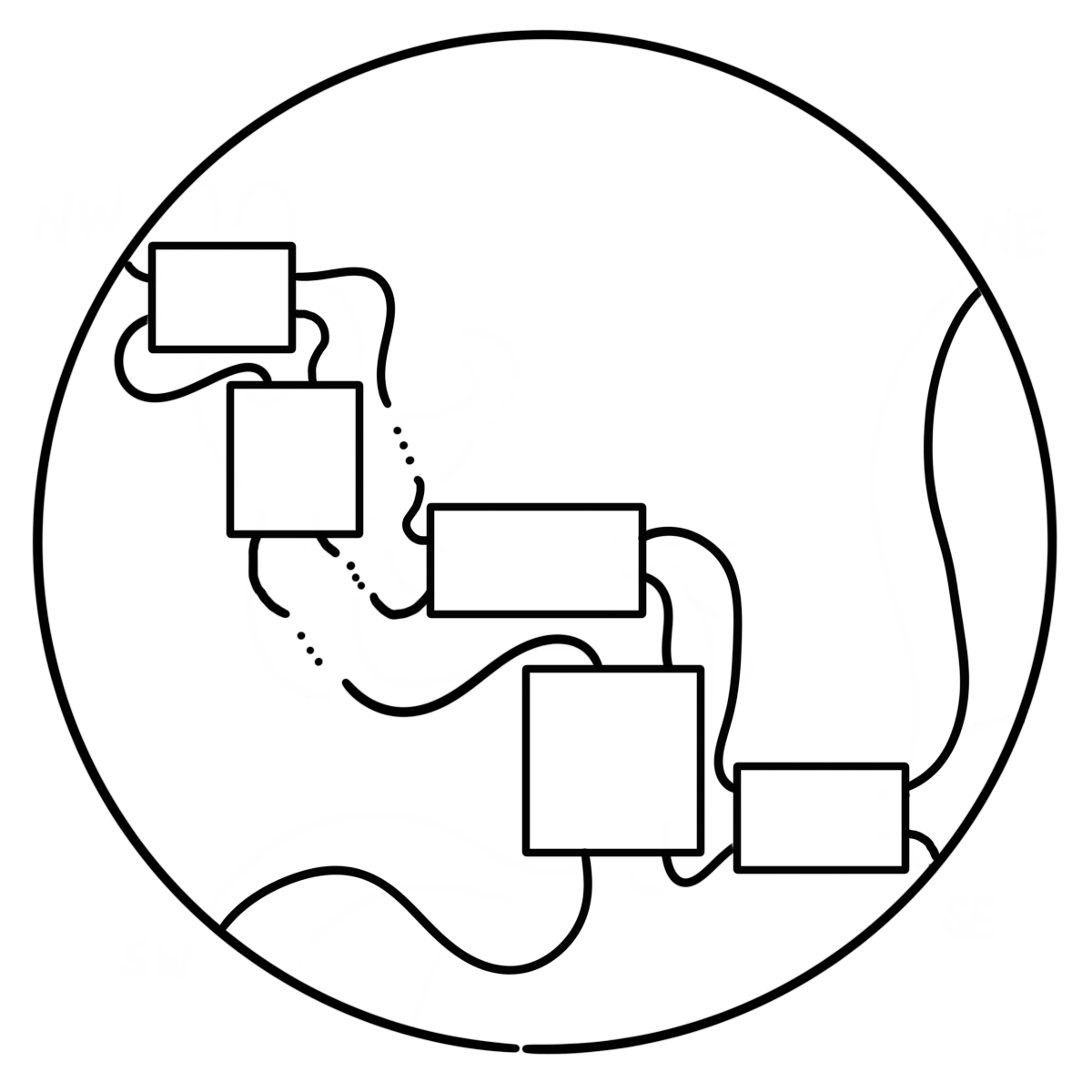}
			\put(73,21.5){$a_n$}
			\put(49.5,25){$a_{n-1}$}
			\put(41.5,46){$a_{n-2}$}
			\put(22,56){$a_2$}
			\put(15,72){$a_1$}
			\put(4,76){$A$}
			\put(92,74){$B$}
			\put(87,14){$C$}
			\put(13,8){$D$}
		\end{overpic}
		\caption{$n$ is odd.}
		\label{fig:odd}
	\end{subfigure}
\medskip

\begin{subfigure}[t]{.4\linewidth}
		\centering
		\begin{overpic}[scale=.13,percent]{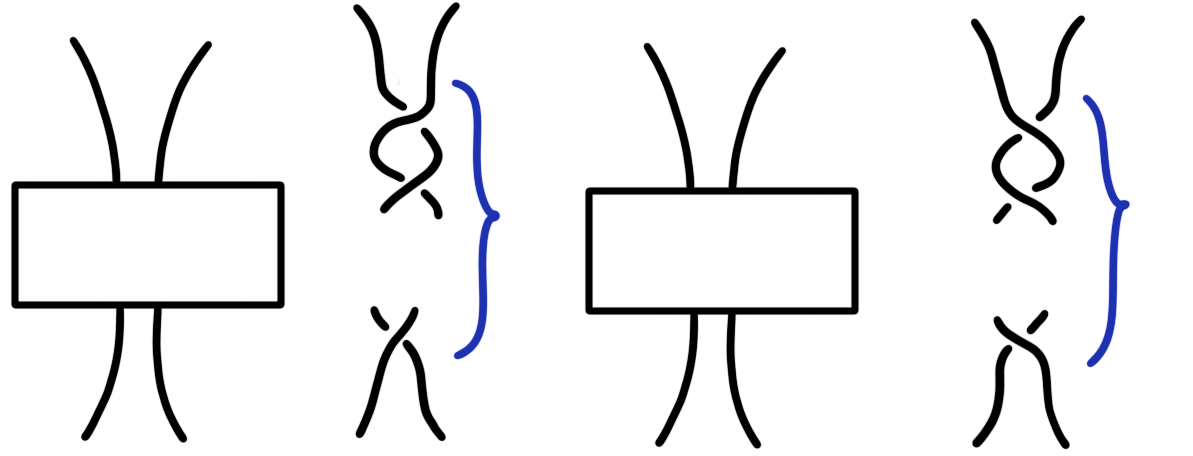}
			\put(5,16){$a>0$}
			\put(42,17){$a$}
			\put(52,15){$a<0$}
			\put(94,15.5){$\minus a$}
		\end{overpic}
		\caption{Vertical crossing.}
		\label{fig:vertical}
\end{subfigure}
\quad
\begin{subfigure}[t]{.4\linewidth}
		\centering
		\begin{overpic}[scale=.13,percent]{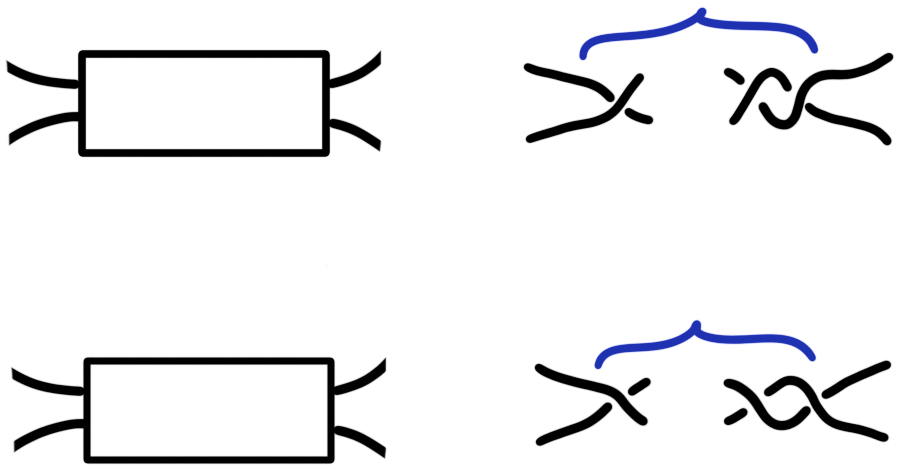}
			\put(13,38){$a>0$}
			\put(13,3){$a<0$}
			\put(80,51){$a$}
            \put(80,15.5){$\minus a$}
		\end{overpic}
		\caption{Horizontal crossing.}
		\label{fig:horizontal}
\end{subfigure} 
\caption{Tangle conventions.}
\end{figure}
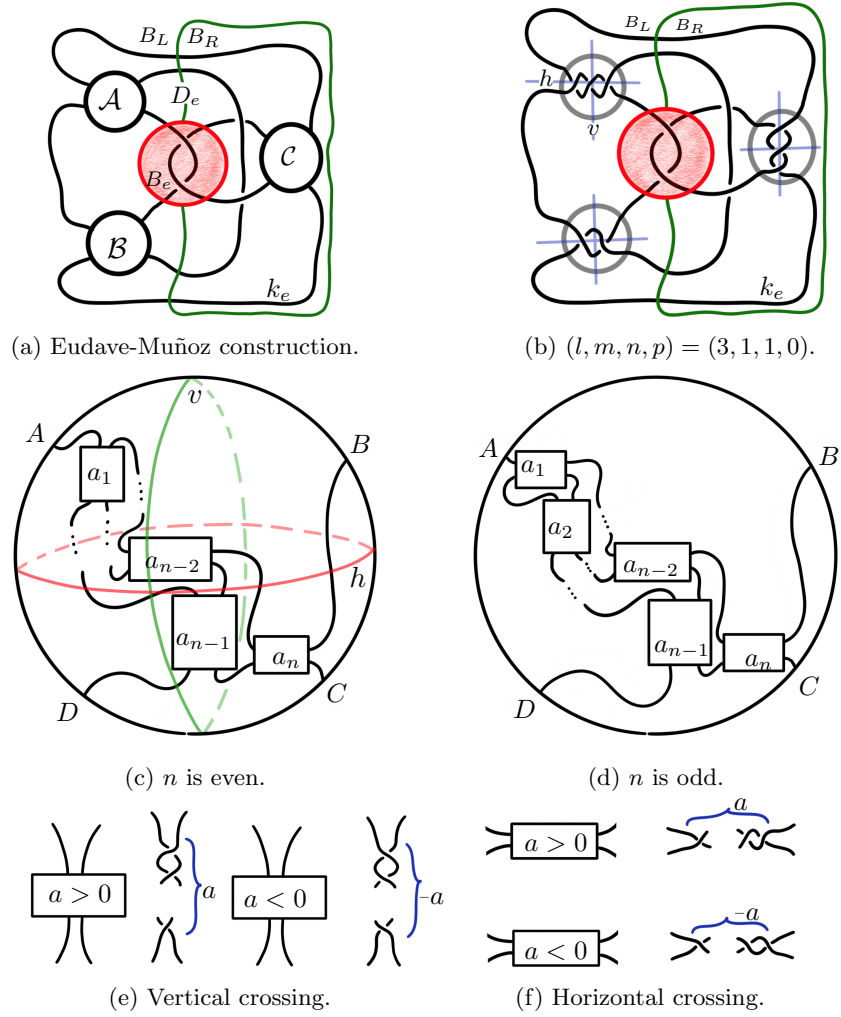 
\subsection{JSJ-decomposition}
Given an atoroidal handlebody-knot $\pair$, the JSJ decomposition theorem asserts that there is, up to isotopy, a unique disjoint union $S$ of essential annuli in $\Compl \HK$, called the \emph{characteristic surface}, such that every component cut off by $S$ is (admissibly) Seifiert/I-fibered or hyperbolic, and the removal of any component of $S$ causes the first condition to fail. An annulus isotopic to a member in $S$ is called a \emph{characteristic annulus}. 

The JSJ-decomposition of an atoroidal handlebody-knot exterior is classified in terms of the JSJ-graph into 14 types by the 
third-named author \cite{Wan:24ii} (see also \cite{KodOzaWan:24}). In particular, we have the following.  
\begin{theorem}[{\cite[Theorem $1.1$, Proposition $2.21$]{Wan:24ii}}]\label{intro:known_jsj} 
Let $S$ be the characteristic surface in the exterior of a non-trivial, atoroidal handlebody-knot. 
Then 
\begin{enumerate}[label=\textnormal{(\roman*)}]
\item there is a unique component $X\subset\Compl\HK$ with $g(\partial X)=2$ cut off by $S$, and it is either I-fibered or hyperbolic;
\item\label{itm:ibundles} if $X$ is I-fibered, then it is I-fibered over a once-punctured annulus or M\"obius band or Klein bottle; particularly, $S$ contains three or two or one annulus, respectively;
\item the other components in $\Compl\HK$ cut off by $S$ are all Seifert fibered solid tori, and adjacent to $X$. 
\end{enumerate}
\end{theorem}
%In addition, if $X$ is I-fibered, then it is I-fibered over a once-punctured annulus or M\"obius band or Klein bottle, and hence $S$ contains three or two or one annulus, respectively.
%The first case in \ref{itm:ibundles} is closely related to non-integral $3$-punctured sphere in a $3$-component link exterior (see Eudave-Ozawa \cite{EudOza:19}), while 
The latter two cases in \ref{itm:ibundles} are intimately connected to handlebody-knots induced by Eudave-Mu\~noz knots. We say an atoroidal handlebody-knot is of \emph{type $M$} (resp.\ \emph{type $K$}) if the JSJ-decomposition of its exterior admits an I-fibered bundle over a once-punctured M\"obius band (resp.\  once-punctured Klein bottle).
\begin{theorem}[{\cite[Theorem $4.3$]{KodOzaWan:24}}]\label{intro:known_typefourone} 
Given an atoroidal handlebody-knot $\pair$, 
if $\pair$ is of type $K$, then $\Compl V$ admits a type $4$-$1$ annulus. 

Conversely,
if $\Compl V$ admits a type $4$-$1$ annulus,  then $\pair$ is either of type $K$ or of type $M$, and furthermore
\begin{enumerate}[label=\textnormal{(\roman*)}]
\item\label{itm:typeK} $\pair$ is of type $K$ if and only if $\pair\simeq \Vr\para$ with 
$l= \pm 2$ or $\Delta(l,m,p)= \pm 2$ or  $\pair\simeq \Vl (\ast,m,n,p)$;
\item\label{itm:typeM} $\pair$ is of type $M$ if and only if $\pair\simeq \Vr\para$ with 
$l\neq \pm 2$ and $\Delta(l,m,p)\neq \pm 2$, 
\end{enumerate}
where $\Delta(l,m,p):=-2lmp+lm+lp+2p-1$. 
\end{theorem}
%Furthermore, by \cite{}, if $\pair$ is of type $K$ or $M$, then every characteristic annulus in $\Compl V$ is of type $3$-$2$;
%we denote by $r_c$ (resp.\ $r_a,r_b$) the slope (resp.\ slopes) of the characteristic annulus (resp.\ annuli) in $\Compl V$ if $\pair$ is of type $K$ (resp.\ type $M$). 

\subsection{Main results} The paper investigates the symmetry group structure of handlebody-knots whose exteriors admit a type $4$-$1$ annuli, and their complement problem. For the symmetry group, we obtain the following classification. 
\begin{theorem}[{Theorems \ref{teo:symmetry_M} and \ref{teo:symmetry_K}}]\label{intro:teo:symmetry}
Suppose $\pair$ admits a type $4$-$1$ annulus.
Then $\sym\pair\simeq \psym\pair$, and
\begin{enumerate}
\item if $\pair$ is of type $M$, then $\sym\pair\simeq \Ztwo$;
\item if $\pair$ is of type $K$, then $\sym\pair$ 
is either $\Ztwo$ or $\Ztwo\times\Ztwo$. 
\item $\sym\pair\simeq \Ztwo\times \Ztwo$ if and only if $\pair$ is equivalent to the left handlebody-knot $\Vl(\ast,1,n,0)$, for some $n$.          
\end{enumerate}
\end{theorem}

A handlebody-knot $\pair$ whose exterior admits a type $4$-$1$ annulus is in general not determined by its exterior $\Compl \HK$, for example, members in the second handlebody-knot family in Lee-Lee \cite{LeeLee:12}, having exteriors homeomorphic to the exterior of $\fivetwo$, all admit a type $4$-$1$ annulus. On the other hand, by \cite{KodOzaWan:24}, every characteristic annulus in $\Compl\HK$ is of type $3$-$2$, and hence has a well-defined slope. Moreover, $\Compl\HK$ can be obtained by gluing two Seifert solid tori (resp.\ one Seifert solid torus) along the characteristic annuli (resp.\ annulus) to an I-bundle over a once-punctured M\"obius band (resp.\ once-punctured Klein bottle). Thus the slope(s) of the characteristic annulus(annuli) in $\Compl\HK$
determines(determine) $\Compl\HK$. In fact, more information is contained in the slope(s), and we have the following Gordon-Lueke type theorem: 
%Let $r_c$ (resp.\ $r_a,r_b$) be the slopes of the characteristic annulus (resp.\ annuli) of $\Compl\HK$ when $\pair$ is of type $K$ (resp.\ $M$).
\begin{theorem}[{Theorems \ref{teo:gordon_luecke_M} and \ref{teo:gordon_luecke_K}}]\label{intro:teo:gordon_luecke}
Suppose the exterior of $\pair$ admits a type $4$-$1$ annulus. Then $\pair$ is determined by
the slope(s) of the characteristic annulus(annuli) of $\Compl \HK$.
 %$r_c$ (resp.\ $(r_a,r_b)$) if $\pair$ is of type $K$ (resp.\ $M$).
%$r_c$ (resp.\ $(r_a,r_b)$) if it is of type $K$ (resp.\ type $M$). 
\end{theorem}

As a corollary of Theorem \ref{intro:teo:gordon_luecke}, we obtain the following.
\begin{corollary}\label{intro:cor:infinite_family}
Given any $m\neq 0,1$, the set $\{\Vl(\ast,m,0,p)\}_{p\in\mathbb{Z}}$ 
is an infinite family of inequivalent handlebody-knots with homeomorphic exteriors.
Particularly, Lee-Lee's second handlebody-knot family \cite{LeeLee:12} is equivalent to
$\{\Vl(\ast,-1,0,p)\}_{p\in\mathbb{Z}}$.
\end{corollary}

%add complement problem

To prove Theorems \ref{intro:teo:symmetry} and \ref{intro:teo:gordon_luecke}, we extend the identities in Eudave-Mu\~noz \cite{Eud:02}, which allow us to refine \cite[Theorem $2.5$(i)]{KodOzaWan:24}.
Recall that by \cite{Wan:24ii}, $\pair$ admits infinitely many essential annuli if and only if it is of type $K$.  
%Hence, $\Compl V$ only has one characteristic annulus. 
\begin{theorem}[Theorem \ref{teo:typeK} and Lemma \ref{lm:typethreetwo_bounds}]\label{intro:teo:number}
Given a handlebody-knot whose exterior admits infinitely many non-characteristic essential annuli, then all but at most five and at least four of them are of type $4$-$1$.
\end{theorem}
Theorem \ref{intro:teo:number} sharpens the result in \cite[Theorem $2.5$(i)]{KodOzaWan:24} by providing the optimal lower bound of non-characteristic annuli not of type $4$-$1$, which is attained by 
$\sixthirdteen$ in the knot table \cite{IshKisMorSuz:12}.
 
%number of type 4-1
%symmetry 

\section{Conventions and Preliminaries}\label{sec:prelim}
%convention for rational tangle
%Eudave-Munoz tangle 
%left right pairs
%%%%%%%%%%%%%%%%%%%%%%
%%%%%%%%%%%%%%%%%%%%%%
%slope computation

%a 2-string tangle
%rational
%chossing two oriented loop
%canonical form
%canonical form determined by a rational number
%[]=\infty
\subsection{Convention} 
Given a subpolyhedron $X$ of a manifold $M$, we denote by $\overline{X}$, 
$\mathring{X}$, $\front X$ and $\mathfrak{N}(X)$, 
the closure, the interior, the frontier, and a regular neighborhood of $X$ in $M$, respectively, and denote by $\vert X\vert$ the number of components in $X$.
The \emph{exterior} $\Compl X$ of $X$ in $M$ is defined to be 
the complement of $\openrnbhd{X}$ in $M$ 
if $X\subset M$ is of positive codimension and to be the closure of $M-X$ otherwise. 
Submanifolds of $M$ are assumed to be proper and in general position.

%The pair 
%$(\sphere,K)$ denotes an embedding of a polyhedron $K$ in $\sphere$, and 
%We denote by $\pair$ a non-trivial, \emph{atoroidal, genus two} handlebody-knot, and assume throughout the paper that $\Compl V$ admits a type $4$-$1$ annulus.

\subsection{Tangle and its diagram}
A $2$-string tangle $(B,t)$ is an embedding of two arcs $t$ in a $3$-ball $B$. A $2$-string tangle is \emph{rational} if $(B,t)$ is homeomorphic to $(D\times I, \{1,-1\}\times I)$, where $D\subset\mathbb{R}^2$ is the disk of radius $2$ with center $(0,0)$.
Two rational tangles $(B,t),(B,t')$ are equivalent if there is a self-homeomorphism of $B$ fixing $\partial B$ and sending $t$ to $t'$.

Fix two loops $v$ and $h$ on $\partial B$ 
such that $v$ separates the endpoints $A, B, C, D$ of $t$ into $A,D$ and $B,C$, and $h$ separates them into $A,B$ and $C,D$, 
with $\vert v\cap h\vert=2$, as shown in Fig.\ \ref{fig:even}. 
%Fix two loops $v,h$ on $\partial B$ that separate the four points $A,B,C,D$ in $t$ with $\vert v\cap h\vert=2$; see Fig.\ \ref{fig:even}. 
A \emph{vertical} (resp.\ \emph{horizontal}) tangle is a tangle
$(B,t)$ where the two components of $t$ are separated by a disk in $B$ bounded by $v$ (resp.\ $h$).
Then the \emph{rational} tangle $R(a_1,\dots,a_n)$ is constructed as follows:
if $n$ is even (resp.\ odd), then we start with a vertical tangle (resp.\ horizontal tangle), and next, twist $A,B$ along $h$ (resp.\ $B,C$ along $v$)
$a_1$ times, and then, twist
$B,C$ along $v$ or $A,B$ along $h$ $a_{i+1}$ times every time after twisting $A,B$ along $h$ or twisting $B,C$ along $v$ $a_i$ times as long as $i+1\leq n$; see Figs.\ \ref{fig:even}, \ref{fig:odd}. We follow the sign convention in \cite{Eud:02}; see Figs.\ \ref{fig:horizontal}, \ref{fig:vertical}. For an example, see Fig.\ \ref{fig:kthreeoneone} given by substituting $\mathcal{A}=R(3),\mathcal{B}=R(1,-3)=R(-2)$, and $\mathcal{C}=R(-1,2,0,2,0)=R(3,0)$ into Fig.\ \ref{fig:emtangle}.

It is known that $R(a_1,\dots,a_n)$ 
is determined by the continued fraction
\[[a_1,\dots,a_n]:=
a_n+\frac{1}{a_{n-1}+\frac{1}{\dots+\frac{1}{a_1}}},\] 
and every rational tangle is equivalent to $R(a_1,\dots,a_n)$ for some $a_1,\dots,a_n\in \mathbb{Z}$.
By convention, %we call the loop separating $a,b$ a vertical loop, 
%and the one separating $b,c$ the horizontal loop, and 
we always project the tangle so that $v$ 
is projected onto the vertical line and $h$ the horizontal line as in Fig \ref{fig:kthreeoneone}, 
and call $v$ (resp.\ $h$) a vertical (resp.\  horizontal) loop of $(B,t)$. 
%with the orientation at ; 
%see Fig.\ \ref{}.

Consider a knot $k$ in $\sphere$ and a $3$-ball $B\subset\sphere$
such that $\mathcal{X}:=(B,B\cap k)$ is rational. Fix a vertical loop $v$ and a horizontal loop $h$ on $\partial B$. Then $\mathcal{X}$ 
is equivalent to $R(a_1,\dots,a_n)$, for some 
$a_1,\dots,a_n\in\mathbb{Z}$. 
%chosen as in Fig.\ \ref{} %by writing this a set of oriented loops is fixed
Let $\pi:M\rightarrow \sphere$ 
be the double branch covering branched along $k$. Then $\pi^{-1}(\sphere-\mathring{B})$ is the exterior of a knot $K$ in $M$ with $U:=\pi^{-1}(B)$ being a regular neighborhood of $K$. 
If $k$ is trivial, then $M=\sphere$, and we call the slopes $r_h,r_v$ of the preimages $l_h:=\pi^{-1}(h),l_v:=\pi^{-1}(v)$ with respect to $K\subset \sphere$ 
% of $h,v$
%under $\pi$ 
%with respect to $U$ 
the \emph{horizontal and vertical slopes} of $\mathcal{X}=R(a_1,\dots,a_n)$, respectively. On the other hand, performing Dehn surgery of slope $r_h,r_v$ on $K$ yields two $3$-manifolds $M_h,M_v$. The orders $\vert H_1(M_h)\vert$, $\vert H_1(M_v)\vert$ of the first integral homology groups are called the \emph{horizontal determinant $d_h$ and vertical determinant $d_v$} of $\mathcal{X}=R(a_1,\dots,a_n)$, respectively.
Set $r_h=\frac{\alpha_h}{\beta_h}$ (resp.\ $r_v=\frac{\alpha_v}{\beta_v}$) with 
$\alpha_h,\beta_h$ (resp.\ $\alpha_v,\beta_v$) relatively prime integers.
Then we have the following.

\begin{lemma}[Mod $\mathbb{Z}$ slope]\label{lm:mod_Z_slope} Modulo $\mathbb{Z}$, we have
\[r_h\equiv(-1)^n[a_n,\dots,a_1]^{-1},\quad r_v\equiv (-1)^n[a_{n-1},\dots,a_1]^{-1}.\]
\end{lemma}
\begin{proof}
Up to sign, $[a_n,\dots,a_1]^{-1}$ (resp.\ $[a_{n-1},\dots,a_1]^{-1}$) computes 
the slope $r_h$ (resp.\ $r_v$), up to some Dehn twist along a meridian disk of $U$. The sign is determined by the convention in Figs.\ \ref{fig:horizontal}, \ref{fig:vertical}.
\end{proof}

\begin{lemma}[Numerator]\label{lm:numerator_slope} Up to sign, $d_h=\alpha_h$ and $d_v=\alpha_v$. 
\end{lemma}
\begin{proof}
$d_h,d_v$ computes how many times $l_h,l_v$ go around the meridian of $U$. 
\end{proof}

%tangle convention
%slope computation 

%%%%%%%%%%%%%%%%%%%%%%%%%%%%%%%%%%%%%%%%%%%%%%%%%%%%%%%%%%%%%%%%%%%%%%%%%%%%%%
\section{Identities for Eudave-Mu\~noz tangles}\label{sec:identities}
%prove the main theorem:
%type K is equivalent to left h.k.
%a refinement of previous result

%\subsubsection{Tangles and its diagram}

In what follows, we will use notations such as $ B_{\rl} $ and $ \mc E_{\rl} (\cdots)$ to collectively refer to the pairs 
$ B_R $ and $ B_L $, and $ \mc E_R (\cdots)$ and $ \mc E_L (\cdots)$, respectively. 
That is, whenever a statement involves $ B_{\rl} $, it is to be understood as holding simultaneously for both $ B_R $ and $ B_L $. 
Likewise, expressions of the form 
$\mc E_{\rl}(\cdots) = \mc E_{\rl}(\cdots)$ are interpreted as abbreviating the conjunction
$\mc E_R(\cdots) = \mc E_R(\cdots)$ and $\mc E_L(\cdots) = \mc E_L(\cdots)$. 

Recall that in the Eudave-Mu\~noz construction, we have 
\begin{itemize}
\item
the double covering $\pi:\sphere\rightarrow \sphere$ branched along $k_e$; 
\item
 the $3$-ball $B_e$, such that the core of the preimage gives us the Eudave-Mu\~noz knot $K\para$; 
 \item
 the two $3$-balls $B_\rl\subset\Compl{B_e}$ cut off by the disk $D_e$; and 
\item the induced right/left handlebody-knots $\Vrl\para$ given by $\pi^{-1}(B_\rl)$. 
\end{itemize}
%Recall that in the Eudave-Mu\~noz construction, we have the double covering $\pi:\sphere\rightarrow \sphere$ branched along $k_e$, and the $3$-ball $B_e$ whose preimage gives us the Eudave-Mu\~noz knot $K\para$, and the two $3$-balls $B_\rl\subset\Compl{B_e}$ cut off by the disk $D_e$, and the induced right/left handlebody-knots $\Vrl\para$ given by $\pi^{-1}(B_\rl)$. 

Denote by $\mc E_\rl\para$ the triplet $(\sphere, B_\rl, k_c)$, and let $\mirrormc E_\rl\para$ (resp.\ 
$\mirrormc V_\rl\para$ and $\mirror K\para$) 
denote the mirror image of 
$\mc E_\rl\para$ (resp.\ 
$\mc V_\rl\para$ and $K\para$).  
It is observed in \cite[Proposition $1.4$]{Eud:97} that 
\begin{align*}
\mirror K(l,m,n,0)&= K(-l,-m,1-n,0), \\
\mirror K(l,m,0,p)&= K(-l,1-m,0,1-p),\\ 
K(\pm l,\pm 1, n, 0)&= K(\mp l\pm 1, \pm 1,n,0).   	 
\end{align*}
These identities can be extended to $\mc E_\rl\para$ as follows.
\begin{lemma}[Mirror Image]\label{lm:mirror}\hfill
\begin{enumerate}[label=\textnormal{(\roman*)}]
\item\label{itm:mirror_n} $\mirrormc E_\rl(l,m,n,0)=\mc E_\rl(-l,-m,1-n,0)$,
\item\label{itm:mirror_p} 
$\mirrormc E_\rl(l,m,0,p)=\mc E_\rl(-l,1-m,0,1-p)$.
\end{enumerate}
\end{lemma}
\begin{proof}
Note first that $\mirrormc E_\rl\para$ can be identified with  
Fig.\ \ref{fig:mirror},
where $\bar{\mathcal{A}}=R(-l), \bar{\mathcal{B}}=R(-p,2,-m,l)$ and 
$\bar{\mathcal{C}}=R(n,-2,1-m,-2,0)$. 
It suffices then to solve $R(l')=R(-l)$, $R(p',-2,m',-l')=R(-p,2,-m,l)$, and 
$R(-n',2,m'-1,2,0)=R(n,-2,1-m,-2,1)$
for $\paraprime$. Applying the identity 
\[
R(\dots,c,\pm 2, d,\dots)=R(\dots,c\pm 1, \mp 2, d\pm 1,\dots),
\] 
we obtain 
$R(n,-2,1-m,-2,1)=R(n-1,2,-m-1,2,0)$ when $p=0$, 
and hence $(l',m',n')=(-l,-m,1-n)$. Similarly, 
when $n=0$, we have $R(-p,2,-m,l)=R(1-p,-2,1-m,l)$ 
and $R(1-m,-2,1)=R(-m,2,0)$, so $(l',m',p')=$ $(-l,1-m,1-p)$.
\end{proof}

\begin{lemma}[Horizontal Flip]\label{lm:hor_flip} 
$\mc E_\rl(\pm l,\pm 1,n,0)=\mc E_\rl(\mp l\pm 1,\pm 1,n,0)$. 
\end{lemma}
\begin{proof}
Rotate $\mc E_\rl\para$ about the horizontal line as in Fig.\ \ref{fig:horizontal_flip}, and solve $R(l')=R(\pm 1, \mp l)=R(\mp l\pm1)$ for $l'$, and then solve $R(m',-l')=R(l)$ for $m'$.
\end{proof}
In addition, we have a generalization of \cite[Proposition $1.4$d]{Eud:97}.
\begin{lemma}[Rotation]\label{lm:pi_rotation}
$\mc E_\rl(2,m,n,0)=\mc E_\lr(2,1-m,0,n)$. 
\end{lemma}
\begin{proof}
Observe first that $\mc E_\rl(2,m,n,0)$ can be identified with Fig.\ \ref{fig:E_twomn}. Then rotating the tangle in Fig.\ \ref{fig:E_twomn} around the center by $\pi$ 
gives us Fig.\ \ref{fig:rotation}. Solving $R(-n',2,m'-1,2,0)=R(-m,2,0)$
and $R(p',-2,m',-l')=R(n,-2,1-m,-2)$ 
for $\paraprime$ gives us the assertion.
\end{proof}
As a corollary, we have the following.
\begin{corollary}\label{cor:right_left}\hfill
\begin{enumerate}[label=\textnormal{(\roman*)}]
\item\label{itm:minustwo} $\mc  E_\rl(-2,m,n,0)=\mc E_\lr(-2,-m,0,n)$,
\item\label{itm:threeone} 
$\mc E_\rl(\pm 3,\pm 1,n,0)=\mc E_\lr(\mp 2,\frac{1\mp 3}{2},0,n)$. 
\end{enumerate} 
\end{corollary}
\begin{proof}
Lemmas \ref{lm:mirror} and \ref{lm:pi_rotation} imply \ref{itm:minustwo}. \ref{itm:threeone} follows from Lemmas \ref{lm:hor_flip} and \ref{lm:pi_rotation}.  
\end{proof}

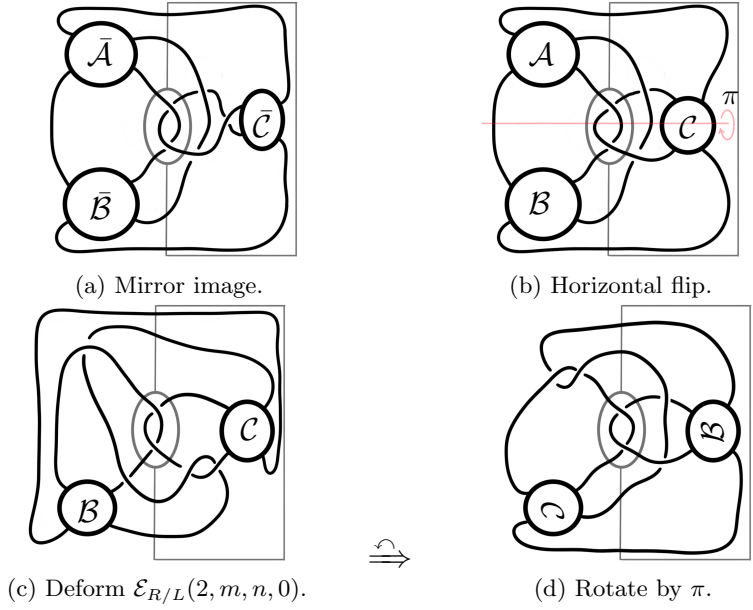
\begin{figure}
\begin{subfigure}{.45\linewidth}
\centering
\begin{overpic}[scale=.12,percent]{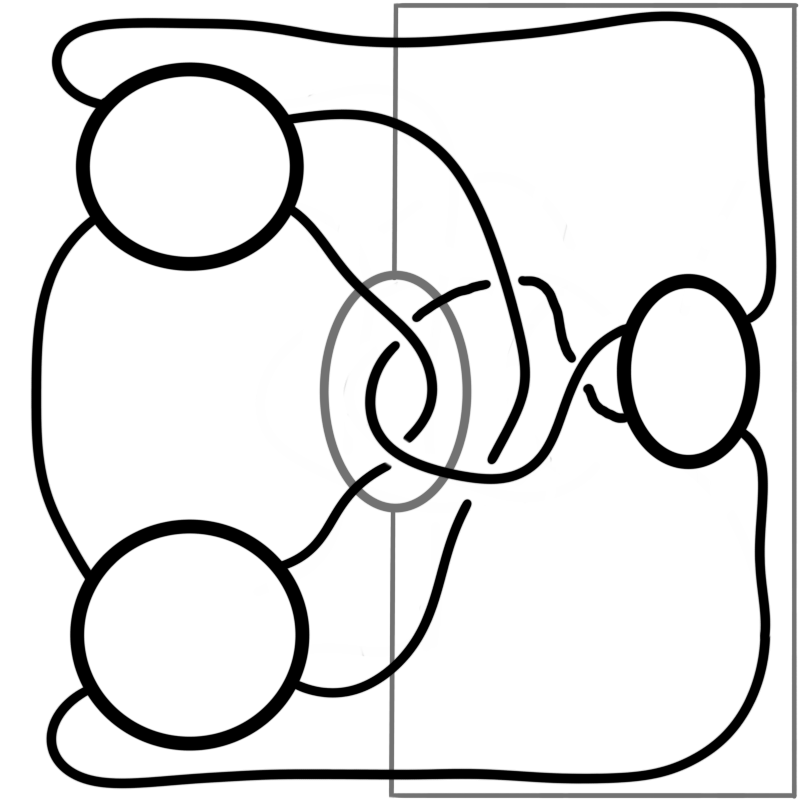}
\put(18,75){\Large $\bar{\mathcal{A}}$}
\put(19,15){\Large $\bar{\mathcal{B}}$}
\put(82.5,47){\Large $\bar{\mathcal{C}}$}
\end{overpic}
\caption{Mirror image.}
\label{fig:mirror}
\end{subfigure}
\begin{subfigure}{.45\linewidth}
\centering
\begin{overpic}[scale=.12,percent]
{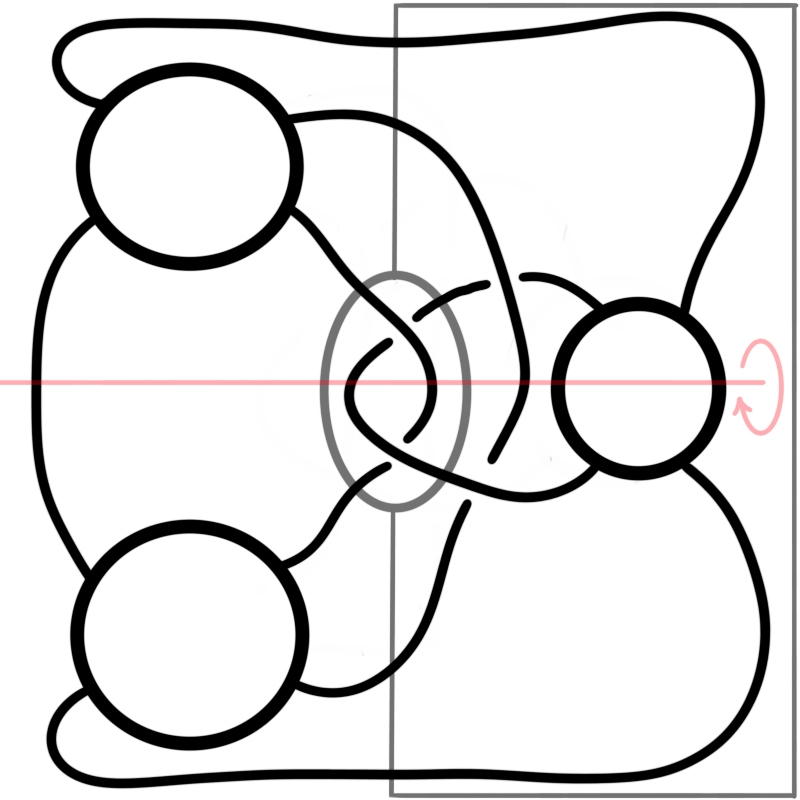}
\put(18,75){\Large $\mathcal{A}$}
\put(18,17){\Large $\mathcal{B}$}
\put(76.5,46.5){\Large $\mathcal{C}$}
\put(93.1,60){\large $\pi$}
\end{overpic}
\caption{Horizontal flip.}
\label{fig:horizontal_flip}
\end{subfigure}
\begin{subfigure}{.42\linewidth}
\centering
\begin{overpic}[scale=.12,percent]{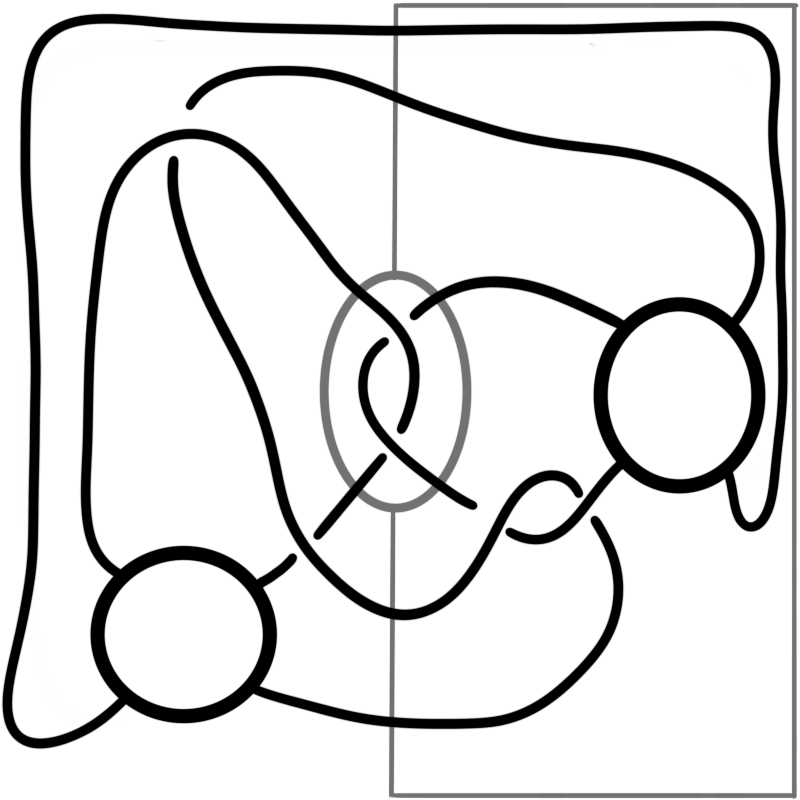}
\put(19,15){\Large $\mathcal{B}$}
\put(82.5,47){\Large $\mathcal{C}$}
\end{overpic}
\caption{Deform $\mc E_\rl(2,m,n,0)$.}
\label{fig:E_twomn}
\end{subfigure}
\raisebox{.5cm}{\Large $\xRightarrow{  \curvearrowleft}$}
\begin{subfigure}{.42\linewidth}
\centering
\begin{overpic}[scale=.12,percent]
{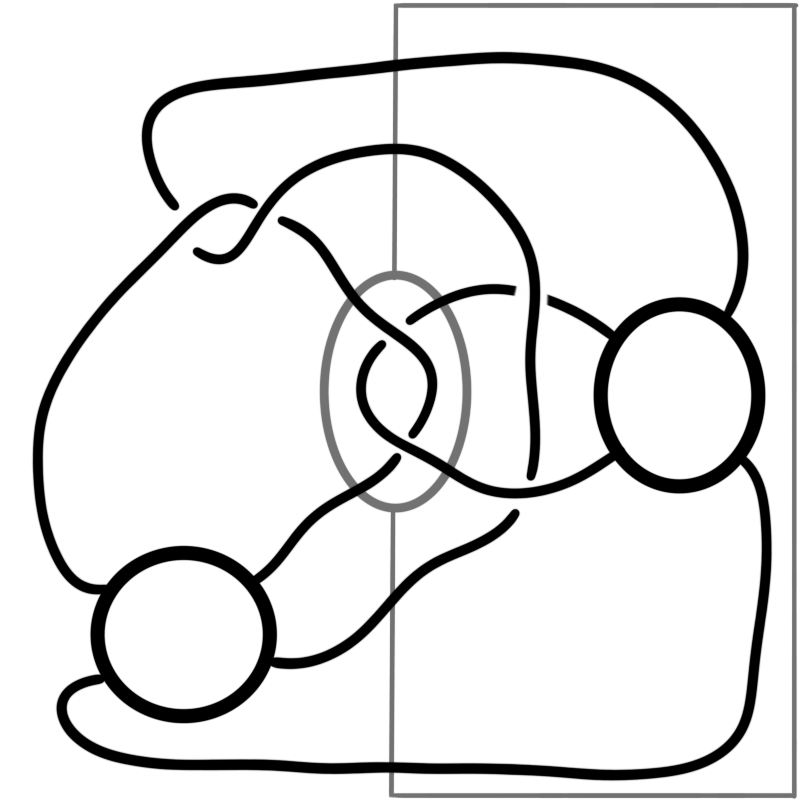}
\put(18.5,15){\rotatebox[origin=c]{-90}{\Large $\mathcal{C}$}}
\put(80.5,46){\rotatebox[origin=c]{90}{\Large $\mathcal{B}$}}
\end{overpic} 
\caption{Rotate by $\pi$.}
\label{fig:rotation}
\end{subfigure}
\caption{Tangle Identities.}
\end{figure}

%%%%%%%%%%%%%%%%%%%%%%%%%%%%%

%%%%%%%%%%%%%%%%%%%%%%%%%%%%%

%In addition, $\Detla(l,m,p)$ 
%\begin{lemma}
%$\Delta(-l,1-m,0,1-p)=-\Delta(l,m,n,p)$, and
%$\Delta(l,m,n,0)=\Delta(-l,-m,1-n,0)$.
%\end{lemma} 
Recall that $\Delta(l,m,p):=-2lmp+lm+pl+2p-1$. For the sake of simplicity, when there is no risk of confusion $l,m,p$ are dropped from the notation $\Delta(l,m,p)$. %Observe also $\Delta=[p,-2,m,-l](2pm-p-m)$. 
We have the following estimate of $\Delta$.
%add estimate of Delta
%add estimate of [p,-2,m] (this can be used later; twice: one for r_a one for r_c)
%   
\begin{lemma}\label{lm:Delta_estimate}
If $p\neq 0,1$ and $l\neq \pm 2$, then 
\[\vert [p,-2,m,-l]\vert \geq\frac{9}{4},\quad %\vert [p,-2,m]\vert\geq \frac{4}{3}, \vert \vert \quad \vert 2p-1\vert \geq 3. 
\vert 2pm-p-m\vert\geq 4, \text{ and } \vert \Delta\vert \geq 9.
\] 
\end{lemma}
\begin{proof}
The constraints: $-l\geq 3, m\leq -1, p\leq -1$, and $-l\leq 3, m\geq 2, p\geq 2$ imply   
\[[p,-2,m,-l]\geq [-1,-2,-1,3]=\frac{9}{4}, \quad [p,-2,m,-l]\leq [2,-2,2,3]=-\frac{9}{4},\]
respectively, and hence the first inequality. The constraints on $m,p$ also imply   
\[[p,-2,m]\geq [2,-2,2]=\frac{4}{3}, \quad [p,-2,m]\leq [-1,-2,-1]=-\frac{4}{3},\]
and hence the second inequality, given $2pm-p-m=[p,-2,m](2p-1)$.   
The last one is given by $\Delta=[p,-2,m,-l](2pm-p-m)$.
\end{proof}

 Recall that an atoroidal handlebody-knot is said to be of type $K$ 
 if the JSJ-decomposition of its exterior admits an $I$-fibered bundle over a once-punctured Klein bottle.

\begin{theorem}\label{teo:typeK}
Given a handlebody-knot $\pair$, then
$\pair$ is of type $K$ if and only if 
$\pair\simeq \Vl\para$, for some $\para$. 
\end{theorem}
\begin{proof}
The direction \lqu $\Leftarrow$ " follows from Theorem \ref{intro:known_typefourone}\ref{itm:typeK}.
%the fact a regular neighborhood of the preimage of $D$ and $D_a$ in Fig.\ \ref{fig:} admits a I-bundle structure over a once-punctured Klein bottle.
%
For the direction \lqu $\Rightarrow$ ", by the first assertion of Theorem \ref{intro:known_typefourone} and Theorem \ref{intro:known_typefourone}\ref{itm:typeK}, it suffices to prove the following:

\noindent
\textbf{Claim: 
If $l=\pm 2$ or $\Delta=\pm 2$, then
$\Vr\para\simeq \Vl\paraprime$, for some $\paraprime$.}

The case $l=\pm 2$ follows from Lemma \ref{lm:pi_rotation} and Corollary \ref{cor:right_left}\ref{itm:minustwo}.
Suppose $\Delta= \pm 2$. If $p=0$, then $lm-1=2$ or $lm-1=-2$;
the former gives us $(l,m)=(\pm 3,\pm 1)$, while the latter has no solution, given $\vert l\vert >1$, so the claim follows from Lemma \ref{cor:right_left}\ref{itm:threeone}.
%assuming l,m,n,p satisfy the EM condition  
If $p\neq 0$, then, given the fact $\Delta(l,m,p)=-\Delta(-l,1-m,1-p)$ and Lemma \ref{lm:mirror}, it suffices to prove the claim for $\Delta=2$. Observe first that 
$\Delta=2$ implies $l\neq \pm 2$. Secondly, Lemma \ref{lm:Delta_estimate} implies $p=1$, so $\Delta(l,m,1)=-lm+l+1=2$, contradicting $l\neq \mp 1$; thus no solution exists.
\end{proof}
The {\color{red} $I$}-bundle over a once-punctured Klein bottle in the exterior of $\Vl\para$ can be visualized as follows \cite[Proof of Theorem $4.3$]{KodOzaWan:24}: 
First, the preimage of the disk $D$ in Fig.\ \ref{fig:em_typeK} under $\pi$ is a type $3$-$3$ annulus $A$, the unique non-separating annulus in the exterior. Secondly, the preimage of the disk $D_l$ in Fig.\ \ref{fig:em_typeK} is a M\"obius band $M_l$. When $l$ satisfies \eqref{eq:constraints} in Theorem~\ref{intro:Eudave_thm}, the frontier $A_l$ of a regular neighborhood of $M_l$ is of type $4$-$1$; otherwise $A_l$ is of type $3$-$2$. 
Since the intersection $A\cap M_l$ 
is an essential arc in $A$ and $M_l$, 
a regular neighborhood of $A\cup M_l$ 
is an {\color{red} $I$}-bundle $X$ over a once-punctured Klein bottle. The other non-characteristic annuli can be produced by smoothing $M_l$ along $A$. In other words, $\{A_l\}_{l\in\mathbb{Z}}$ is the set of separating non-characteristic annuli in the exterior. Note also the frontier of 
$X$ can be identified with the preimage of 
$D_c$ in Fig.\ \ref{fig:em_typeK} under $\pi$ and is the unique characteristic annulus.

\cout{
As a corollary of Theorem \ref{teo:typeK}, the following sharpens \cite[Theorem $2.5$(i)]{KodOzaWan:24}.
\begin{corollary} 
If $\pair$ is of type $K$, then $\Compl V$ admits either four or five type $3$-$2$ annuli, and one type $3$-$3$ annulus, and all the other annuli are of type $4$-$1$. 
\end{corollary}
\begin{proof}
Given \cite[Theorem $2.5$]{KodOzaWan:24} and the characteristic annulus being of type $3$-$2$, it suffices to show that there are at least three $l$'s such that $A_l$ is of type $3$-$2$, and they are $l=0,\pm 1$ by \eqref{eq:constraints}. 
\end{proof}

By \eqref{eq:constraints}, one can make Corollary \ref{cor:typethreetwo_bounds} more precise.
}
\begin{lemma}\label{lm:typethreetwo_bounds}
If $\pair\simeq \Vl(\ast,1,n,0)$ 
or $\Vl(\ast,2,0,1)$ (resp.\ $\Vl(\ast,-1,n,0)$).
then $A_l$ is of type $3$-$2$ if and only if $l=0,\pm 1, 2$ (resp.\ $l=0,\pm 1,-2$);
otherwise $A_l$ is of type $3$-$2$ if and only if $l=0,\pm 1$.
\end{lemma} 
\begin{proof}
This follows from the fact that 
$(l,m,p)=(2,1,0)$, $(-2,-1,0)$, or 
$(2,2,1)$ is not allowed in \eqref{eq:constraints}. 
%and the fact 
%that $\Vl(2,2,0,1)\simeq \mirrormc V_L(-2,-1,0,0)\simeq \Vl(2,1,1,0)$ by Lemma \ref{lm:mirror}. 
\end{proof}
%Lemma \ref{lm:typethreetwo_bounds}, together with Theorem \ref{teo:typeK}, sharpens \cite[Theorem $2.5$(i)]{KodOzaWan:24} by providing an optimal lower bound.

\begin{figure}[h]
	\begin{subfigure}[t]{.49\linewidth}
		\centering
		\begin{overpic}[scale=.16,percent]{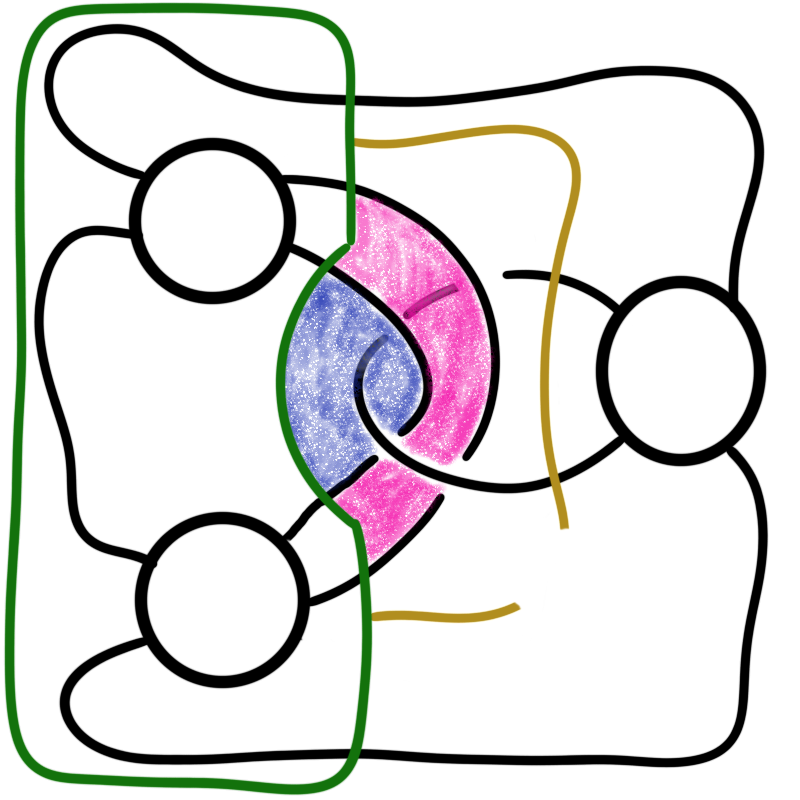}
			\put(82,51){\large $\mathcal{C}$}
			\put(21,70){\large $\mathcal{A}$}
			\put(24,20){\large $\mathcal{B}$} 
			\put(33,91.5){\footnotesize $B_L$}
			\put(46,67){\footnotesize $D$}
			\put(37,55){\footnotesize $D_l$}
			\put(66.5,25.5){$D_c$} 
		\end{overpic}
		\caption{$I$-bundle in the exterior of $\Vl\para$.}
		\label{fig:em_typeK}
	\end{subfigure}
	\begin{subfigure}[t]{.49\linewidth}
		\centering
		\begin{overpic}[scale=.16,percent]
						{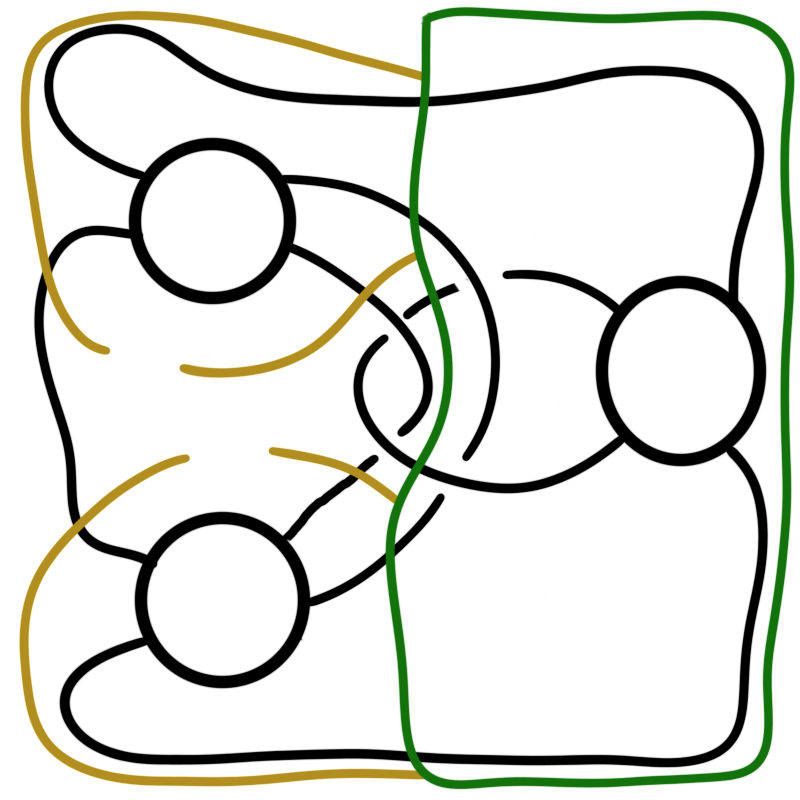}
			\put(82,51){\large $\mathcal{C}$}
			\put(21,70){\large $\mathcal{A}$}
			\put(24,20){\large $\mathcal{B}$}
			\put(12.5,52){\large $D_a$} 
			\put(23,41){\large $D_b$}
	    \end{overpic}
		\caption{$I$-bundle in the exterior of $\Vr\para$.}
		\label{fig:em_typeM}
	\end{subfigure}
	\caption{}
\end{figure}

\section{Symmetry classification}\label{sec:classification}
%computation of slope for A, for B
%classification for type K
%classification for type M
Consider the tangles $\mathcal{A},\mathcal{B}$ in Fig.\ \ref{fig:emtangle}, and
denote by $\hslopeA,\vslopeA$ (resp.\ $\hslopeB,\vslopeB$) the horizontal and vertical slopes of the tangle $\mathcal{A}$ (resp.\ $\mathcal{B}$), and by $\hdetA, \vdetA$ (resp.\ $\hdetB, \vdetB$) the horizontal and vertical determinants of $\mathcal{A}$ (resp.\ $\mathcal{B}$). Recall that $\Delta:=-2lmp+lm+lp+2p-1$, and
for the sake of simplicity, we set $\MN:=4mn-2m+1$ and
$\MP:=2pm-p-m$; also $\para$ satisfies the constraints in \eqref{eq:constraints}, and denote by $\MN_0$ (resp.\ $\MP_0$ and $\Delta_0$) the value of $\MN$ (resp.\ $\MP$ and $\Delta$) when plugging in $n=0$ (resp.\ $p=0$). In particular, $\MN_0=-2m+1,\MP_0=-m, \Delta_0=ml-1$. Note that $\Delta,\MP$ are related so that
\begin{equation}\label{eq:delta_mp_relation}
\Delta=-l\MP+2p-1. 
\end{equation}

%add a figure and refer to it
Recall from \cite{KodOzaWan:24}, if $\pair$ is of type $M$, then the two characteristic annuli $A_a,A_b$ are the preimages of $D_a,D_b$ in Fig.\ \ref{fig:em_typeM} under $\pi$. The slopes $r_a,r_b$ of $A_a,A_b$ are precisely $\hslopeA,\hslopeB$, respectively. 
Similarly, if $\pair$ is of type $K$, then the slope pair $(r_1,r_2)$ of the type $3$-$3$ annulus $A$ is $(\vslopeA,\vslopeB)$; see Fig.\ \ref{fig:em_typeK}.
%add slope pair at the intro 

\begin{lemma}\label{lm:LMNP_constraints}
We have $\MP\neq 0$, $\Delta\neq 0\neq \MP$, and $\vert \MN\vert >2$.
\end{lemma}
\begin{proof}
Suppose $\MP=0$. Then $p\neq 0$, for otherwise, we would have $m=0$ contradicting \eqref{eq:constraints}. Since $p\neq 0$, we have $n=0$ 
and therefore $m\neq 0,1$. Now, $\MP=0$ implies that $p(m-1)=-m(p-1)$, so $m\vert p$ and $p\vert m$. In other words, $m=\pm p$. If $m=p$ (resp.\ $m=-p$), then $\MP=0$ implies $2m(m-1)=0$ (resp.\ $-2m^2=0$), an impossibility, so $\MP\neq 0$.

Suppose $\Delta=0$. Consider first the case $p=0$. 
Then $\Delta=0$ is equivalent to $ml-1=0$, contradicting $l\neq \pm 1$. If $p\neq 0$, then $\Delta=(lm-1)(1-2p)+lp=0$, and hence $(lm-1)(2p-1)=lp$.
This implies $l=2p-1,p=lm-1$ or $l=1-2p,p=1-lm$. 
Substituting $lm-1$ (resp.\ $1-lm$) for $p$ in $l=2p-1$ (resp.\ $l=1-2p$), we obtain $l(2m-1)=3$ (resp.\ $l(2m-1)=1$). Since $l\neq \pm 1$ and $m\neq 0,1$, neither occurs, so $\Delta\neq 0$.
 
To see the last assertion, we note first that $\MN$ is odd, and $\MN=\pm 1$ if and only if $m=0$ or $(m,n)$ is $(1,0)$ or $(-1,1)$, contradicting the constraints \eqref{eq:constraints}     
\end{proof}

\begin{lemma}[Slope modulo $\mathbb{Z}$]\label{lm:mod_Z_slopes_AB} Modulo $\mathbb{Z}$, we have
\[\hslopeA\equiv -\frac{1}{l} 
\ ,
\quad  
\vslopeA\equiv 0
\ , \text{ and }\quad
\hslopeB\equiv \frac{-2(ml-1)+l}{\Delta} 
\ ,
\quad  
\vslopeB\equiv \frac{2m-1}{\MP}
\ .
\]  
\end{lemma}
\begin{proof}
The assertion follows from Lemma \ref{lm:mod_Z_slope} as we have 
\begin{multline*}
\hslopeA\equiv [l]^{-1}(-1),\vslopeA\equiv[\phi]^{-1}(-1),\\  
\hslopeB\equiv [-l,m,-2,p]^{-1}(-1)^4,\text{ and } \vslopeB\equiv [m,-2,p]^{-1}(-1)^4 \quad (\text{ mod } \mathbb{Z}\ ). 
\end{multline*} 
\end{proof}
Let $\beta^\mathcal{A}_h,\beta^\mathcal{A}_v,\beta^\mathcal{B}_h,\beta^\mathcal{B}_v$ be the denominators of 
$\hslopeA,\vslopeA,\hslopeB,\vslopeB$, respectively.
\begin{corollary}\label{cor:den_slopeAB}
Up to sign, we have 
\begin{align*}
\beta^\mathcal{A}_h &=l, \\
\beta^\mathcal{A}_v &=1, \\ 
\beta^\mathcal{B}_h &=\Delta=2p-1-l \MP=(ml-1)(1-2p)+pl\\ 
\beta^\mathcal{B}_v &= \MP.   
\end{align*}
\end{corollary}

\begin{lemma}[Determinants]\label{lm:det_AB}
Up to sign, 
\begin{align*}
\hdetA &=-l\MN\MP+4pn-1,\\ 
\vdetA &=\MN\MP,\\
\hdetB &=2-4n+l\MN=(4n-2)(lm-1)+l,\\   \vdetB &=\MN.
\end{align*}
\end{lemma}
\begin{proof}
Observe that the $3$-manifold obtained by performing Dehn surgery along a lifting of the horizontal (resp.\ vertical) loop of $\mathcal{A}$ is the double branched cover branched along the Montesinos link in Fig.\ \ref{fig:montesino_link} (resp.\ the link in Fig.\ \ref{fig:link_sum}) with $\mathcal{X}=\mathcal{B}$ and $\mathcal{C}'$ in Fig.\ \ref{fig:Cprime}. Therefore, up to sign, $\hdetA$ is the determinant 
\begin{multline*} 
\begin{vmatrix}
	-2&0&0&1\\
	0&\Delta&0&\MP\\
	0&0&1-2n&2mn-m-n+1\\
	1&1&1&0
\end{vmatrix}
\\
=2\MP(1-2n)+2(2mn-m-n+1)\Delta-(1-2n)\Delta 
=2\MP(1-2n)+\MN\Delta\\
\overset{\eqref{eq:delta_mp_relation}}{=}2\MP(1-2n)+\MN(-l\MP+(2p-1))
=-l\MN\MP+(2p-1)\MN+2\MP(1-2n)\\
=-l\MN\MP+4pn-1.
\end{multline*}

Since \ref{fig:link_sum} is a sum of two rational links, up to sign, $\vdetA$ is the product of the denominator of the denominator of $[-n,2,m-1,2,0]$ and the denominator of $[p,-2,m,-l]$, which is  
$(4mn-2m+1)(2pm-p-m)=\MN\MP.$

Similarly, the manifold obtained by performing Dehn surgery along a lifting of the horizontal (resp.\ vertical) loop of $\mathcal B$ is a branched cover of $\sphere$ along the Montesinos link in Fig.\ \ref{fig:montesino_link}
(resp.\ Fig.\ \ref{fig:link_sum}) with $\mathcal{X}=\mathcal{A}$. Thus up to sign, $\hdetB$ is the determinant:
\begin{multline*} 
\begin{vmatrix}
	-2&0&0&1\\
	0&l&0&1\\
	0&0&1-2n&2mn-m-n+1\\
	1&1&1&0
\end{vmatrix}
\\
=2(1-2n)+2l(2mn-m-n+1)-l(1-2n)
=2-4n+l\MN,
\end{multline*} 
%Since $\MN_0=-2m+1$, $2-4n+l\MN_0=-2lm+l+2$, and thus the assertion for $\hdetB$. 
and up to sign, $\vdetB$ is the denominator of $[-n,2,m-1,2,0]$.
\end{proof}

\begin{figure}
\begin{subfigure}[b]{.36\linewidth}
\centering
\begin{overpic}[scale=.12,percent]{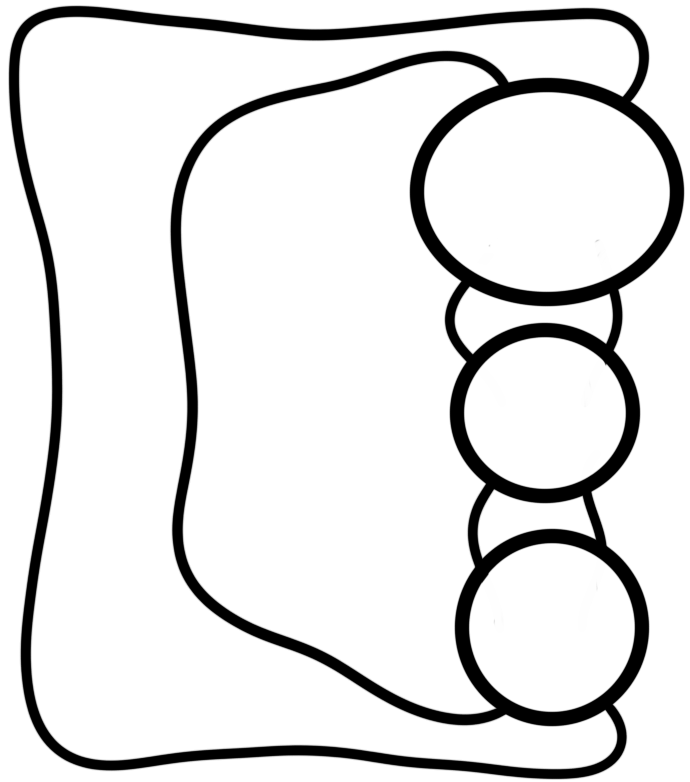}
\put(56.5,72.5){$R(\minus 2)$}
\put(65.2,42.5){\Large $\mathcal{X}$}
\put(67,15.5){\Large $\mathcal{C}'$}
\end{overpic}
\caption{Montesinos Link.}
\label{fig:montesino_link}
\end{subfigure} 
\begin{subfigure}[b]{.31\linewidth}
\centering
\begin{overpic}[scale=0.1,percent]{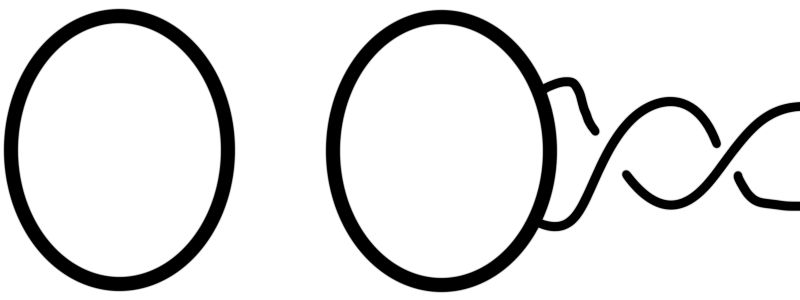}
\put(10,13){\Large $\mathcal{C}'$}
\put(30,15){\large $=$} 
\put(51,13){\rotatebox[origin=c]{-90}{\Large $\mathcal{C}$}} 
\end{overpic}
\caption{$\mathcal{C}'=R(n,-2,1-m,0)$.}
\label{fig:Cprime}
\end{subfigure}
\begin{subfigure}[b]{.31\linewidth}
\centering
\begin{overpic}[scale=.12,percent]{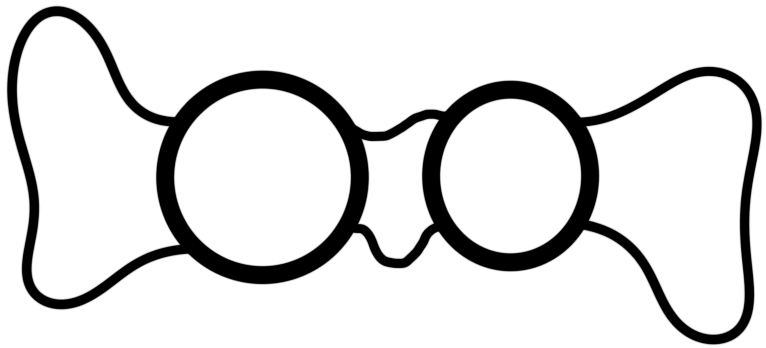}
\put(30.5,17.5){\Large $\mathcal{C}$}
\put(62,17.5){\Large $\mathcal{X}$} 
\end{overpic}
\caption{Sum of rational links.}
\label{fig:link_sum}
\end{subfigure} 
\caption{}
\end{figure}

\begin{lemma}[Horizontal slopes]\label{lm:horizontal_slope_AB}
If $l\neq \pm 2$ and $\Delta\neq \pm 2$, then
\begin{align*}
\hslopeA & =-\MN\MP-\dfrac{1}{l} , \\
\hslopeB & =\dfrac{2-4n+l\MN}{2p-1-l\MP}=\dfrac{(ml-1)(4n-2)+l}{(ml-1)(1-2p)+pl}. 
\end{align*}
\end{lemma}
\begin{proof}
	Up to sign, the denominators and numerators of $\hslopeA$ and $\hslopeB$ are given by Corollary \ref{cor:den_slopeAB} and Lemma \ref{lm:det_AB}. Since $\vert l\vert>2, \vert \Delta\vert>2$, the signs of $\hslopeA,\hslopeB$ can be
    determined by Lemma \ref{lm:mod_Z_slopes_AB}.  
\end{proof}

To compute the vertical slopes, we observe the following properties of $\Delta,\MN, \MP$.
\begin{lemma}\label{lm:additional_MP_MN}
\begin{align}
(2-4n)\MP+(2p-1)\MN&=\pm 1,\label{eq:MP_MN}\\ 
(2-4n)+l\MN&\neq 0\label{eq:MN} {\color{red} . }
\end{align}
\end{lemma}
\begin{proof}
The first assertion follows from the fact that, if $p=0$, \eqref{eq:MP_MN}$= (2-4n)(-m)+(-1)(4mn-2m+1)=1$, and if $n=0$, \eqref{eq:MP_MN}$= 2(pm-p-m)+(2p-1)(-2m+1)=-1$.

For the second assertion, we note first that \eqref{eq:MN}$=(lm-1)(4n-2)+l$. Suppose $(lm-1)(4n-2)+l=0$. 
Then when $n=0$, we have $l(2m-1)=2$, contradicting that $l\neq \pm 1,m\neq 0,1$ by \eqref{eq:constraints}. If $n\neq 0$, then $p=0$. Since $lm-1,l$ are relatively prime, we have $lm-1=\pm 1$, yet this implies $(l,m)=(\pm 2,\pm 1),(\pm 1,\pm 2)$ or $l=0$ or $m=0$, but none occurs by \eqref{eq:constraints}, given $p=0$. This proves the second assertion.
\end{proof}
We also need the following simple lemma.
\begin{lemma}\label{lm:intersection}
Let $u,v,x,y$ be four non-zero integers. 
If 
\[
\begin{vmatrix}
u&v\\
x&y
\end{vmatrix}
=\pm 1,  
\]
then 
\[
\begin{vmatrix}
-u&v\\
x&y
\end{vmatrix}
\neq\pm 1 .  
\]
\end{lemma}
\begin{proof}
Suppose otherwise. Then either $uy=0$ or $vx=0$, a contradiction. \end{proof}

\begin{lemma}[Vertical Slopes]\label{lm:vertical_slope_AB}
\begin{align*}
\vslopeA & =-\MN\MP , \\
\vslopeB & =-\dfrac{\MN}{\MP}. 
\end{align*}
\end{lemma}
\begin{proof}
By Corollary \ref{cor:den_slopeAB} and Lemma \ref{lm:det_AB}, up to sign, 
$\vslopeA =-\MN\MP$ and 
$\vslopeB =-\dfrac{\MN}{\MP}$.
Now, a lifting of the vertical loop and a lifting of the horizontal loop of $\mathcal{A}$ (resp.\ $\mathcal{B}$) meet at a point. 
Thus the sign is correct by Lemma \ref{lm:intersection} since 
\begin{multline*}
\begin{vmatrix}
-l\MN\MP-1& l\\
-\MN\MP& 1
\end{vmatrix}
=-1
\\
 \left( \text{resp.}\ 
\begin{vmatrix}
2-4n+l\MN & 2p-1-l\MP\\
-\MN & \MP
\end{vmatrix}
=(2-4n)\MP - (2p-1)\MN
=\pm 1
 \right)
\end{multline*}
by Lemma \ref{lm:additional_MP_MN} and none of $-l\MN\MP-1$, $l$, $\MN,\MP,\Delta=2p-1-l\MP$, $2-4n+l\MN$ is zero by Lemma \ref{lm:LMNP_constraints}.  
\end{proof}

%%%%change the notation.
\begin{theorem}\label{teo:slopes_A1_A2} 
If $\pair$ is of type $M$, then $\pair\simeq \Vr \para$ for some $\para$, and 
\begin{align*}
(r_a,r_b)& =\Big(\dfrac{lm\MN-1}{l},\ \dfrac{(4n-2)(lm-1)+l}{lm-1}\Big)\ ,\  p=0\ ; \\
(r_a,r_b)&=\Big(
 \dfrac{l(2m-1)\MP-1}{l},\
\dfrac{-2(lm-1)+l}{\Delta}\Big)\ ,\ n=0.
\end{align*} 
\end{theorem}
\begin{proof}
Plug in $p=0$ (resp.\ $n=0$) into 
the formula of $\hslopeA,\hslopeB$ in Lemma \ref{lm:horizontal_slope_AB}.
\end{proof}

\begin{corollary}\label{cor:ra_rb_not_equal}
If $\pair$ is of type $M$, then $r_a\neq r_b$. 
\end{corollary}
\begin{proof}
Observe first that the denominator and numerator in each fraction are relatively prime. Suppose $r_a=r_b$. 
If $p=0$, then we have $\pm l=lm-1$, and hence $l(m\mp 1)=1$, contradicting $l\neq \pm 1$. If $n=0$, then 
we have 
\[
l(2m-1)\MP-1=\pm 2(lm-1)\mp l=\pm l(2m-1)\mp 2,
\] 
and hence $l(2m-1)(\MP\mp 1)=1\mp 2$, 
but neither can happen since both $\vert l\vert, \vert 2m-1\vert$ are greater than $1$, given $l\neq 0,\pm 1$ and $m\neq 0,1$.   
\end{proof}

\begin{theorem}\label{teo:slope_pair}
If $\pair$ is of type $K$, then $\pair\simeq  \Vl \para$, for some $\para$, and 
%If $\pair$ is of type $K$, then $\pair\simeq \mf H_L\para$, for some $\para$, and 
%the slope pair    
\[
(r_1,r_2)=
\begin{dcases}
\Big(\MN m\ ,\ \dfrac{\MN}{m}\Big) {\color{red} , } &\ p=0 ;\\
\Big((2m-1)\MP\ ,\ \dfrac{2m-1}{\MP}\Big) {\color{red} , }&\ n=0.
\end{dcases}
\]
\end{theorem}
\begin{proof}
It follows from Lemma \ref{lm:vertical_slope_AB} as the slope pair is given by $(\vslopeA,\vslopeB)$. 
\end{proof}

\begin{figure}
%\begin{subfigure}{.7\linewidth}
\centering
\begin{overpic}[scale=.2,percent]{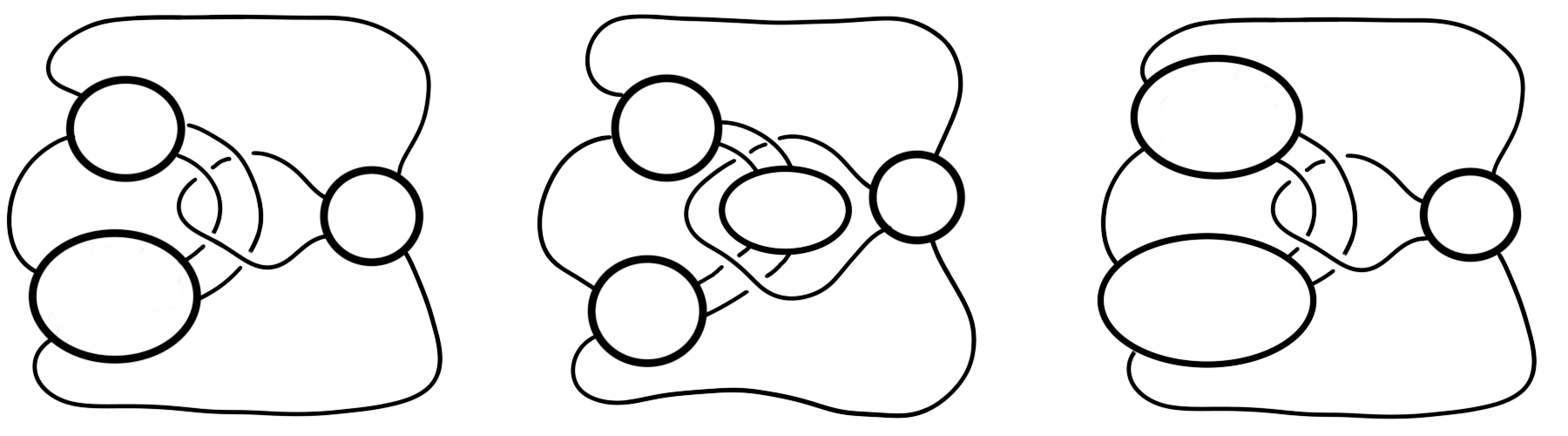}
\put(2.85,8){\footnotesize $R(\minus\ast,\plusminus 1)$}
\put(5.5,19.1){\small $R(\ast)$}
\put(23,13){$\mathcal{C}$}
\put(28.3,12.5){\huge $\rightsquigarrow$}
\put(38,7.3){\small $R(\minus\ast)$}
\put(40,19.1){\small $R(\ast)$}
\put(58,14){$\mathcal{C}$}
\put(47.4,13.7){\footnotesize $R(\plusminus 1)$}
\put(63,12.5){\huge $\rightsquigarrow$}
\put(74,19.7){\small $R(\ast\plusminus 1)$}
\put(70.8,7.7){\footnotesize $R(\minus\ast\minusplus 1,\plusminus 1)$}
\put(93,13){$\mathcal{C}$}
\end{overpic}
\caption{From $\mc V_L(\ast,\pm 1,n,0)$ to $\mc V_L(\ast\pm 1,\pm 1, n, 0)$.}
\label{fig:swapping}
%\end{subfigure} 
%\begin{subfigure}{.29\linewidth}
%\centering
%\begin{overpic}[scale=.15,percent]{vdet.png}
%\end{overpic}
%\end{subfigure}
\end{figure}
\begin{corollary}\label{cor:slope_pair}
If $\pair$ is of type $K$,  
%and $(r_1,r_2)$ is the slope pair of $A$, 
then the following are equivalent:
\begin{enumerate}[label=\textnormal{(\roman*)}]
\item\label{itm:symmetric_pair} $r_1=r_2$.
\item\label{itm:symmetric_left} $\pair\simeq \Vl(\ast,\pm 1,n,0)${\color{red}.}
%at some point we should mention the independency of l
\item\label{itm:swapping} there exists a self-homeomorphism of $\pair$ swapping the boundary components of the type $3$-$3$ annulus $A$ in $\Compl\HK$.
\end{enumerate}
\end{corollary} 
\begin{proof} 
\ref{itm:symmetric_pair} $\Rightarrow$ \ref{itm:symmetric_left}:  
By Theorem \ref{teo:typeK}, it may be assumed $\pair\simeq \Vl\lpara$, so it suffices to show that 
$\Vl\lpara\simeq \Vl(\ast,\pm 1,n',0)$, for some $n'$. 
By Theorem \ref{teo:slope_pair} and Lemma \ref{lm:LMNP_constraints}, if $p=0$, then 
$r_1=r_2$ implies $m=\pm 1$; hence 
$\Vl\lpara\simeq \Vl(\ast,\pm 1,n,0)$.
If $p\neq 0$, then $n=0$, and $r_1=r_2$ implies $\MP=2pm-p-m=\pm 1$, or equivalently $p(2m-1)=m\pm 1$. 
Since $2m-1, m-1$ are relatively prime, if 
$p(2m-1)=m-1$, then $2m-1=\pm 1$, contradicting 
$m\neq 0,1$. The same argument applies to the case $p(2m-1)=m+1$, so
the greatest common divisor of $2m-1, m+1$ must be $3$. In particular, $2m-1=\pm 3$, or equivalently, $m=-1,2$. If $m=-1$, then $m+1=0$ and hence $p=0$, contradicting the assumption. If $m=2$, then $p=1$. Applying Lemma \ref{lm:mirror}, we see $\Vl(\ast,2,0,1)\simeq \mirrormc V_l(-\ast,-1,0,0)\simeq \Vl(\ast,1,1,0)$.

\ref{itm:symmetric_left} $\Rightarrow$ \ref{itm:swapping}: It may be assumed that 
$\pair\simeq \Vl(0,\pm 1, n,0)$. Observe that
there are two homeomorphisms from $\Vl(0,\pm 1, n, 0)$ to $\Vl(\pm 1, \pm 1, n, 0)$: one moves a half twisting in $\mathcal{B}$ to $\mathcal{A}$ along the two strings connecting $\mathcal{A},\mathcal{B}$ as illustrated in Fig.\ \ref{fig:swapping}, while the other is the horizontal flip given in Lemma \ref{lm:hor_flip}; see Fig.\ \ref{fig:horizontal_flip}. 
The composition of the two induces a self-homeomorphism of $\pair$ swapping the two components of $\partial A$.
 
The direction \ref{itm:swapping} $\Rightarrow$ \ref{itm:symmetric_pair} is clear.  
\end{proof}

\begin{theorem}[Type $M$]\label{teo:symmetry_M}
If $\pair$ is of type $M$, %\pair assumed to have a type 4-1 
then 
\[\sym\pair\simeq \psym\pair\simeq \Ztwo.\]
\end{theorem}
\begin{proof}
Note first it may be assumed that $\pair\simeq \Vr\para$ with $l\neq \pm 2$ and $\Delta\neq \pm 2$ by Theorem \ref{intro:known_typefourone}. Also by \cite[Theorem $5.2$]{KodOzaWan:24}, $\Compl \HK$ admits a unique the type $4$-$1$ annulus, so $\sym\pair\simeq \psym\pair$. 
Now, since $A_a,A_b$ are the only characteristic annuli in $\Compl\HK$, given $\pair$ is of type $M$, Corollary \ref{cor:ra_rb_not_equal} implies 
\[\sym\pair\simeq \mcg{\sphere,\HK,A_1,A_2}\simeq \mcg{\sphere,\HK,B_a,B_b},\] 
where $B_a,B_b$ are the annuli in $\partial\HK$ cut off by $\partial A_a,\partial A_b$, respectively. 
Let $P:=\partial\HK-cB_a\cup cB_b$, where $cB_a$ and $cB_b$ denote the cores of the annuli $B_a$ and $B_b$, respectively. Then we have the commutative diagram:
\begin{equation}\label{eq:diagram}
%\begin{center}
\begin{tikzpicture}[baseline=(current  bounding  box.center)]
%\node (pair) at (-3.25,0) {$\psym\pair$};
\node (pairB) at (0,0) {$\pmcg{\sphere,\HK,B_a,B_b}$};
\node (boundaryB) at (4,0) {$\pmcg{\partial\HK,B_a, B_b}$};
\node (boundaryb) at (4,-1) {$\pmcg{\partial\HK,\{[cB_a],[cB_b]\}}$};
\node (boundary) at (4,-2) {$\pmcg{\partial\HK}$.};
\node (B) at (7.5,0) {$\pmcg{B_a, B_b}$};
\node (P) at (7.5,-1) {$\pmcg{P}$};
%\node (dot) at (7.5,-2) {.};
%\draw[->] (pair) to node [above]{$\simeq$}(pairB);
\draw[->] (pairB) to (boundaryB);
\draw[->] (boundaryB) to (B);
\draw[->] (boundaryb) to node[above]{$\lambda$} (P);
\draw[->] (boundaryB) to (boundaryb);
\draw[->] (boundaryb) to (boundary);
\draw[->] (pairB) to[out=10,in=170] node[above]{$\psi$} (B); 
\draw[->] (pairB) to[out=-80,in=180] node[above]{$\iota$} (boundary); 
\end{tikzpicture}
%\end{center}
\end{equation}
Given a mapping class $[f]\in\pmcg{\sphere,V,B_a,B_b}$, suppose $\psi([f])$ is trivial. Then 
$f\vert_P$ does not permute the four punctures of $P$.  Note that the pure mapping class group of $P$ is free
\cite[Section $4.2.4$]{FarMar:12}, so $[f]$, and hence $[f\vert_P]$, being of finite order implies $[f\vert_P]=1$. 

This implies that $[f\vert_{\partial\HK}]$ 
is in the kernel of the cutting homomorphism $\lambda$ in \ref{eq:diagram}. The kernel, being generated by Dehn twists along $cB_a,cB_b$, is free \cite[Proposition $3.20$]{FarMar:12}, so $[f\vert_{\partial\HK}]$ is trivial in $\pmcg{\partial\HK}$; thereby $[f]$ is trivial in $\pmcg{\sphere,\HK,B_a, B_b}$, given $\iota$ is injective. Consequently, $\psi$ is injective, so $\sym\pair\simeq \image\psi<\pmcg{B_a,B_b}\simeq \Ztwo\times\Ztwo$.

%Now, by Corollary \ref{cor:slopes_A1_A2}, 
%the slopes of $A_1,A_2$ are different, so
%no mapping class in $\image\psi$ swaps $B_1,B_2$,
%and hence $\image\psi<\Ztwo\oplus\Ztwo$. 
To see $\image\psi=\Ztwo$, we recall that 
$A_a\cup A_b$ cut off an $I$-bundle $X$ over a once-punctured M\"obius band from $\Compl\HK$. 
Since every $f\in\phomeo{\sphere,\HK,A_1,A_2}$ preserves the lids of $X$, 
%
%In addition, 
%Then there is a basis $\langle u,v\rangle=H_1(M)$ such that the images of $H_1(A_a)\rightarrow H_1(M)$ and 
%$H_1(A_b)\rightarrow H_1(M)$ are $u$ and 
%$u+2v$, respectively. 
the restrictions of $f$ on $cA_a,cA_b$, as well as on $cB_a,cB_b$, are either both orientation-preserving or both orientation-reversing, so $\image\psi<\Ztwo$. 

On the other hand, the involution $g$ of $\pair$ that induces the branched covering $\pi:(\sphere,\HK)\rightarrow (\sphere,B_R)$ induces a mapping class $[g\vert_{B_a\cup B_b}]$ that reverses the orientation of both $cB_a,cB_b$, so $\image\psi\simeq\Ztwo$.
\end{proof}

\begin{theorem}[Type $K$]\label{teo:symmetry_K}
If $\pair$ is of type $K$, 
then 
\[\Ztwo<\sym\pair\simeq \psym\pair<\Ztwo\times\Ztwo,\]
and $\sym\pair\simeq \Ztwo\times \Ztwo$ if and only if  $\pair\simeq \Vl(\ast,\pm 1,n,0)$.
\end{theorem}
\begin{proof}
%By \cite[Lemma $4.9$]{Wan:23pp}, 
By Theorem \ref{intro:known_jsj}\ref{itm:ibundles} that $\pair$ admits a unique characteristic annulus, the preimage of $D_c$ under $\pi$ in Fig.\ \ref{fig:em_typeK}, and it is of type $3$-$2$, so $\sym\pair\simeq \psym\pair$.
By Theorem \ref{teo:typeK}, $\pair\simeq \Vl\para$, for some $\para$, and the involution that induces the branched covering $(\sphere,\Vl\para)\rightarrow (\sphere, B_L)$ 
implies that $\Ztwo<\sym\pair$.

Now, there is a unique disk in $\HK$ disjoint from the unique type $3$-$3$ annulus $A$, which induces a spine $\Gamma_A$ of $\HK$, and hence the isomorphisms:
\begin{equation}\label{eq:isomorphisms}
\psym\pair\simeq\pmcg{\sphere,\HK,A}\simeq \pmcg{\sphere,\Gamma_A}.
\end{equation} 
Since $\pair$ is atoroidal and non-trivial, the natural homomorphism 
\begin{equation}\label{eq:cho_koda}
\pmcg{\sphere,\Gamma_A}\rightarrow \mcg{\Gamma_A}
\end{equation} 
is injective by \cite[Theorem $2.5$]{ChoKod:13}. The composition of
\eqref{eq:isomorphisms} and \eqref{eq:cho_koda} gives the injection: 
\begin{equation}\label{eq:injection}
\iota:\psym\pair\simeq\pmcg{\sphere,\Gamma_A}\rightarrow \mcg{\Gamma_A}.
\end{equation}  

On the other hand, the restriction map induces the homomorphism 
\[
\phi:\pmcg{\sphere,\HK,A}\rightarrow \mcg{A}\simeq\Ztwo\times\Ztwo.
\] 
Observe that, if $\phi([f])$ is trivial, then $\iota([f])\in \mcg{\Gamma_A}$ is trivial, and 
therefore $[f]$ is trivial by \eqref{eq:injection}. As a result, $\phi$ is also an injection. 

Furthermore, $\phi$ is surjective if and only if there exists a self-homeomorphism of
$\pair$ swapping the two components of $\partial A$, yet this happens if and only if $r_1=r_2$. The assertion thus follows from Corollary \ref{cor:slope_pair}. 
\end{proof} 

\section{Complement problem}\label{sec:complement}
Here we show that, for a handlebody-knot that admits a type $4$-$1$ annulus, the slopes of its characteristic annuli determine the handlebody-knot.

%let $r_a,r_b$ be the slopes of the characteristic annuli corresponding to 
%$\mathcal{A},\mathcal{B}$, respectively. 
\begin{lemma}[$r_a$ estimate]\label{lm:ra_estimate}
Suppose $\Vr\para$ is of type $M$. 
%If $n\neq 0$, then 
\[ 
\text{If $n\neq 0$, then }
\begin{dcases}
%	\begin{align}
	\vert r_a\vert \geq \frac{14}{3}, &n\neq 1;\\
	\vert r_a\vert \geq \frac{8}{3}, &n= 1.	
%	\end{align}
\end{dcases}
\]
%if $p\neq 0$, then 	
\[ 
\text{If $p\neq 0$, then }
\begin{dcases}
	\vert r_a\vert \geq \frac{41}{3}\ , &p\neq 1;\\
	\vert r_a\vert \geq \frac{8}{3}\ , &p= 1.	
\end{dcases}
\]
\end{lemma}
\begin{proof} 
Consider first the case: $n\neq 0$. This implies $p=0$ and $r_a=m(4mn-2m+1)-\frac{1}{l}$. 
%it suffices to show $\vert m(4mn-2m+1)\vert\geq 3$ when $n=1$ and
%$\vert m(4mn-2m+1)\vert\geq 5$ when $n\neq 1$, given $\vert l\vert\geq 3$.
If $n=1$, then $m\geq 1$ or $m\leq -2$ by \eqref{eq:constraints}, and hence $m(4mn-2m+1)=(2m+1)m\geq 3$. The assertion then follows from the constraints $\vert l\vert \geq 3$, given $\Vr\para$ being of type $M$. 
Now, observe that  	 
\[[2,m-1,2,-n]=\frac{-4mn+2m-1}{4m}.\]
If $n\neq 1$, then $m\geq 1$ or $m\leq -1$ by \eqref{eq:constraints}, so 
\[\dfrac{5}{4}=[2,0,2,1]\leq \frac{-4mn+2m-1}{4m}\quad \text{or}\quad \frac{-4mn+2m-1}{4m}\leq [2,-2,2,-2]=-\dfrac{5}{4}.\]
The assertion thus follows from 
\[\vert m(4mn-2m+1)\vert\geq \dfrac{5}{4}(4m^2) \geq 5.\]

Consider the case $p\neq 0$, which implies $n=0$ and 
$r_a=(2m-1)(2pm-p-m)-\frac{1}{l}$.  
%therefore it suffices to show that 
%$\vert (2m-1)(2pm-p-m)\vert \geq 3$ when $p=1$ 
%and $\vert (2m-1)(2pm-p-m)\vert \geq 12$ when $p\neq 1$. 
By \eqref{eq:constraints}, we have $m\geq 2$ or $m\leq -1$, and therefore $\vert 2m-1\vert\geq 3$. 
If $p=1$, then the assertion follows from 
$(2m-1)(2pm-p-m)=(2m-1)(m-1)\geq 3$.   

If $p\neq 1$, then by Lemma \ref{lm:Delta_estimate}, we have $\vert (2m-1)(2pm-p-m)\vert\geq 12$, and thus the assertion.  
\end{proof}
\begin{lemma}[$r_b$ estimate]\label{lm:rb_estimate}
Suppose $\Vr\para$ is of type $M$. If $p\neq 0$, then 
\[ 
	\begin{dcases}
		\vert r_b\vert \leq \frac{7}{9}\ , & p\neq 1;\\
		\frac{7}{5}\leq r_b \leq \frac{7}{2}\ , &p= 1.	
	\end{dcases}
\]
\end{lemma}
\begin{proof}
	Note first that $r_b=[-l,m,-2,p,0]$, and we have the constraints: $\vert l\vert \geq 3$, and $m\geq 2$ or $m\leq -1$.
	If $p=1$, then 
	\[\dfrac{7}{5}=[3,-1,-2,1,0]\leq r_b\leq [-3,2,-2,1,0]=\dfrac{7}{2}\] 
	If $p\neq 1$, then $p\geq 2$ or $p\leq -1$, so  
	\[-\dfrac{7}{9}=[3,-1,-2,-1,0]\leq r_b\leq [-3,2,-2,2,0]=\dfrac{7}{9}.\]
\end{proof}

Set $\tilde r_b:=[-l,m,0]=\frac{l}{lm-1}$.
\begin{lemma}\label{lm:tilderb_estimate}
Suppose $\Vr\para$ is of type $M$. 
Then 
\[ 
\begin{dcases}
	\vert \tilde r_b\vert \leq \frac{3}{5}\ , &\vert m\vert\geq 2;\\
	\dfrac{3}{4}\leq \tilde r_b \leq \frac{3}{2}\ , &m= 1;\\
	-\dfrac{3}{2}\leq \tilde r_b \leq -\frac{3}{4}\ , &m= -1.	
\end{dcases}
\]
\end{lemma}
\begin{proof}
If $m\neq \pm 1$, then $\vert m\vert \geq 2$, so 
\[[3,-2,0]\leq \tilde r_b\leq [-3,2,0].\] 
The case $m=\pm 1$ follows from 
$\vert l\vert \geq 3$.
\end{proof}

\begin{theorem}\label{teo:gordon_luecke_M}
Let $\mc V,\mc V'$ be two type $M$ handlebody-knots. Then 
$\mc V\simeq \mc V'$ if and only if 
the slopes of their characteristic annuli are the same. 
\end{theorem}
\begin{proof}
The ``only if " direction is clear. 
For the ``if" direction, note that it may be assumed that $\mc V=\Vr\para$ and $\mc V'=\Vr\paraprime$ for some $\para$ and $\paraprime$ with $l,l',\Delta(l,m,p),\Delta(l',m',p')$ not equal to $\pm 2$ by Theorem \ref{intro:known_typefourone}\ref{itm:typeM}. Denote by $r_a,r_b$ (resp.\ $r_a',r_b'$) the slopes of the characteristic annuli of $\mc V$ (resp.\ $\mc V'$) corresponding to 
$\mathcal{A},\mathcal{B}$, respectively.

\subsection*{Case 1: $r_a=r_a',r_b=r_b'$.}
Since $\vert l\vert,\vert l'\vert$ are greater than $2$, $r_a=r_a'$ 
implies $l=l'$.
\subsection*{Case 1.1: $p=p'=0$.}  
Observe first that  
\[ 
r_b=4n-2+\dfrac{l}{lm-1}=r_b'=4n'-2+\dfrac{l}{lm'-1},
\]
implies that either $lm-1=lm'-1$ or $lm-1=-lm'+1$. The latter implies $l(m'+m)=2$, contradicting $\vert l\vert>2$. The former implies
$m=m'$ and hence $n=n'$, so $\mc V\simeq \mc V'$. 
\subsection*{Case 1.2: $n=n'=0$.} 
%As in the previous case, $r_a=r_a'$ implies 
%that $l=l'$, while by 
%Since $r_b=r_b'$, 
%This implies $n=n'=0$.
Comparing the numerators of $r_b,r_b'$ gives us either
\[-2(lm-1)+l=-2(lm'-1)+l \quad \text{or} \quad 
-2(lm-1)+l=2(lm'-1)-l.\] 
The latter yields $l(m'+m'-1)=2$, contradicting $\vert l\vert>2$. The former implies $m=m'$, and thus comparing the denominators of $r_b,r_b'$, we obtain \[
(lm-1)(1-2p)+pl=(lm-1)(1-2p')+p'l,\] 
equivalent to 
$p(l-2lm+2)=p'(l-2lm+2)$.
Since $l-2lm+2=l(1-2m)+2\neq 0$, given $\vert l\vert>2$, we have $p=p'$, so $\mc V\simeq \mc V'$.
\subsection*{Case 1.3: $n\neq 0,p'\neq 0$.} 
This implies $p=0,n'=0$, and hence $r_b=r_b'$ gives us 
\[
4n-2+\dfrac{l}{lm-1}=\dfrac{-2(lm'-1)+l}{(lm'-2)(1-2p')+p'l}.
\]
If $n\neq 1$, then by Lemma \ref{lm:tilderb_estimate}, $\vert r_b\vert=\vert 4n-2+\tilde r_b\vert\geq \frac{9}{2}$, contradicting Lemma \ref{lm:rb_estimate}. We thus conclude $n=1$, and hence $m\neq 0,-1$ by \eqref{eq:constraints}. This implies, by Lemma \ref{lm:tilderb_estimate}, 
$-\frac{3}{5}\leq \tilde{r_b}\leq \frac{3}{2}$, and hence    
%\[-\dfrac{3}{5}\leq \dfrac{l}{lm-1}\leq \dfrac{3}{2}\] 
%In particular, we have  
\[\dfrac{7}{5}\leq r_b=r_b'+2\leq \dfrac{7}{2}.\]
This, together with Lemma \ref{lm:rb_estimate} applied to $r_b'$, implies $p'=1$.
Plugging in $n=1,p'=1$ into $r_b=r_b'$ yields 
\[2+\dfrac{l}{lm-1}=
%\dfrac{-2(lm'-1)+l}{-(lm'-1)+l}=
2+\dfrac{l}{lm'-1-l},\]
and therefore $lm-1=lm'-1-l$, or equivalently, $l(m-m'+1)=0$, so $m'=m+1$. Thus $\mc V\simeq \Vr(l,m,1,0), \mc V'\simeq \Vr(l,m+1,0,1)$, and by Lemma \ref{lm:mirror}\ref{itm:mirror_n}\ref{itm:mirror_p}, we obtain  
\[\Vr(l,m,1,0)\simeq \mirrormc V_R(-l,-m,0,0) \simeq \Vr(l,m+1,0,1),\]
so $\mc V\simeq \mc V'$.
\subsection*{Case 2: $r_a=r_b',r_b=r_a'$}
\subsection*{Case 2.1: $p=p'=0$} 
The identity $r_a=r_b'$ gives us 
\[m(4mn-2m+1)-\dfrac{1}{l}=4n'-2+\dfrac{l'}{l'm'-1}.\]
Consider first the case $\vert m'\vert \geq 2$. Thus by Lemma \ref{lm:tilderb_estimate}, $\vert \frac{l'}{l'm'-1}\vert \leq \frac{3}{5}$. Since $\vert \frac{1}{l}\vert \leq \frac{1}{3}$, we have 
\[\vert\dfrac{1}{l}+\dfrac{l'}{l'm'-1}\vert<1,\ \text{so}\ -\dfrac{1}{l}=\dfrac{l'}{l'm'-1},\] 
yet this contradicts $l'\neq \pm 1$. Thus we conclude $m=\pm 1$. The same argument applies to $r_b=r_a'$, and yields $m'=\pm 1$. 
Since $m'=\pm 1$, modulo $\mathbb{Z}$,  
\[ 
\frac{1}{-l}\equiv \frac{l'}{\pm l'-1}\equiv
\frac{1}{l'\mp 1}. 
\]
Since $\vert \frac{1}{l}\vert\leq \frac{1}{3}, \vert \frac{1}{ l'\mp 1}\vert \leq \frac{1}{2}$, we have $\frac{1}{-l}=\frac{1}{l'\mp 1}$, so
$-l=l'\mp 1$.
By symmetry, $m=\pm 1$ implies that $-l'=l\mp 1$. 
In particular, $(m,m')$ is either $(1,1)$ or $(-1,-1)$. In each case, solving the equation $r_a=r_b'$ (resp.\ $r_b=r_a'$) yields $n=n'$.
As a result, $\mc V\simeq \Vr(l,\pm 1,n,0),\mc V'\simeq \Vr(-l\pm 1,\pm 1, n,0)$, and by Lemma \ref{lm:hor_flip}, $\mc V\simeq \mc V'$.
\subsection*{Case 2.2: $n=n'=0$}
We divide it into two subcases: both $p,p'$ are nonzero or exactly one of $p,p'$ is nonzero. For the first case, by Lemmas \ref{lm:ra_estimate}, \ref{lm:rb_estimate}, $r_a=r_b'$, $r_b=r_a'$ happen only when $p=p'=1$. Thus,
$r_a=(2m-1)(m-1)-\frac{1}{l}$. If $m\geq 3$ or $m\leq -1$, then $r_a\geq 5$, yet $r_b'\leq \frac{7}{2}$ in all cases by
Lemma \ref{lm:rb_estimate}. This implies $m=2$. The same reasoning applies to $r_a'=r_b$ and yields $m'=2$. Plugging in $m=m'=2$ and $p=p'=1$ into $r_a=r_b'$ yields 
\[3-\dfrac{1}{l}=\dfrac{-2(2l'-1)+l'}{-(2l'-1)+l'}=3+\dfrac{1}{l'-1},\]
and hence 
%gives 
%$3-\dfrac{1}{l}=\dfrac{-2(2l'-1)+l'}{-(2l'-1)+l'}=3+\dfrac{1}{l'-1}$, and hence 
$l'=-l+1$. Consequently, $\mc V\simeq \Vr(l,2,0,1), \mc V'\simeq \Vr(-l+1,2,0,1)$. By Lemma \ref{lm:mirror}\ref{itm:mirror_p} and Lemma \ref{lm:hor_flip}, we have 
\[\Vr(l,2,0,1)\simeq \mirrormc V_R(-l,-1,0,0)\simeq \mirrormc V_R(l-1,-1,0,0)
\simeq \mc V_R(-l+1,2,0,1),\]
and thus $\mc V\simeq \mc V'$.

If one of $p,p'$, say $p'$, is nonzero, 
then we have 
\[r_a=m(1-2m)-\frac{1}{l}=r_b'=\frac{-2(l'm'-1)+l'}{(l'm'-1)(1-2p')+p'l'}.\]
Since $m\geq 2$ or $\leq -1$, we have $r_a\leq -\frac{8}{3}$, contradicting $r_b'\geq -\frac{7}{9}$ by Lemma \ref{lm:rb_estimate}.
\subsection*{Case 2.3: $n\neq 0,p'\neq 0$}
%This implies $p=0,n'=0$, and hence 
By Lemmas \ref{lm:ra_estimate}, \ref{lm:rb_estimate} and $r_a=r_b'$, we have
$n=1,p'=1$. Thus $r_a=m(2m+1)-\frac{1}{l}$, so 
$r_a> 5$ when $\vert m\vert \geq 2$;
however $r_b'\leq \frac{7}{2}$ by Lemma \ref{lm:rb_estimate}; therefore $m=1$.
Similarly, to determine $m'$, we observe that  
\[\dfrac{7}{5}=[3,-1,-2,1,0]\leq r_b'=[-l',m',-2,p',0]\leq [-3,3,-2,1,0]=\dfrac{13}{5}
\]
when $m'\geq 3$ or $m'\leq -1$, contradicting $r_b'=r_a\geq \frac{8}{3}$ by Lemma \ref{lm:ra_estimate}. We conclude $m'=2$.
Plugging in $(m,n)=(1,1)$ and $(m',p')=(2,1)$ into $r_a=r_b'$ yields 
\[3-\dfrac{1}{l}=3+\dfrac{1}{l'-1},\] 
which implies $l'=-l+1$, and 
$\mc V\simeq \Vl(l,1,1,0), \mc V'\simeq \Vl(-l+1,2,0,1)$. 
By Lemmas \ref{lm:mirror}, \ref{lm:hor_flip}, 
we have 
\[\Vr(l,1,1,0)\simeq \mirrormc V_R(-l,-1,0,0)\simeq \mirrormc V_R(l-1,-1,0,0)\simeq \Vr(-l+1,2,0,1)\]
and hence $\mc V\simeq  \mc V'$.
%%need both lemmas
\end{proof}

Let $r_c$ be the slope of the characteristic annulus, the preimage of $D_c$ under $\pi$ in Fig.\ \ref{fig:em_typeK}. 
\begin{lemma}\label{lm:slope_rc}
\[r_c=-\dfrac{8pm-4p-4m}{4mn-2m+1}.\]
\end{lemma}
\begin{proof}
By Lemma \ref{lm:mod_Z_slope}, module $\mathbb{Z}$, 
\begin{equation}\label{eq:modZ_rc}
r_c\equiv (-1)[2,m-1,2,-n]^{-1}=\dfrac{4m}{4mn-2m+1}=\dfrac{4m}{\MN} .
\end{equation}  
By Lemma \ref{lm:numerator_slope}, up to sign, the numerator of $r_c$ is the order of the first homology group of the double branched cover branched along the Montesinos link in Fig.\ \ref{fig:montesions_C}, which can be computed by the determinant 
\begin{equation}\label{eq:numerator_rc}
\begin{vmatrix}
-2&0&0& 1\\
0&2&0& 1\\
0&0&2p-1& \MP\\
1&1&1&0
\end{vmatrix}
=4\MP.
\end{equation} 
In particular, $r_c$ is $\frac{4\MP}{\MN}$, up to sign. Since $\vert \MN\vert >2$ by Lemma \ref{lm:LMNP_constraints}. \eqref{eq:modZ_rc} implies that $r_c=-\frac{4\MP}{\MN}$.
\end{proof}
\begin{figure}
%{.31\linewidth}
\centering
\begin{overpic}[scale=0.1,percent]{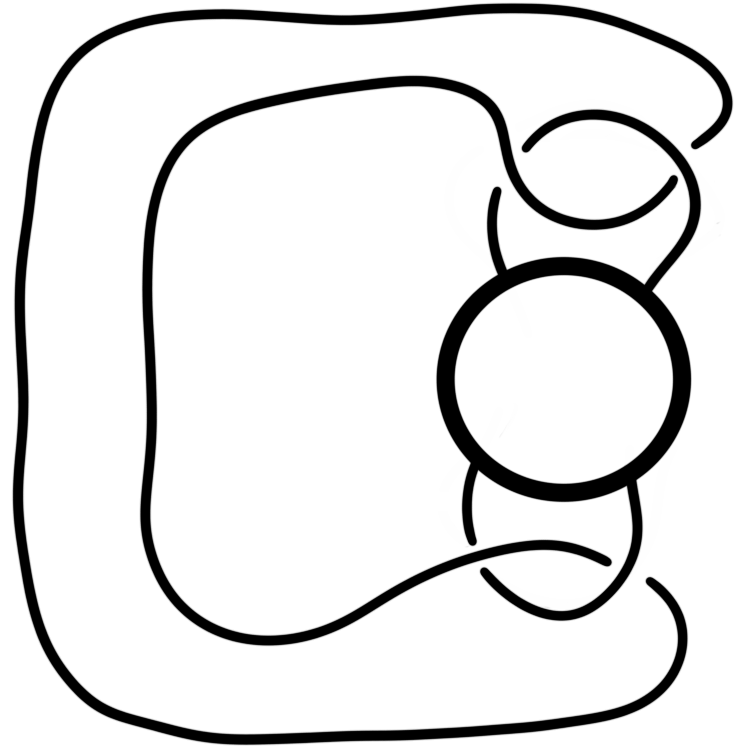}
\put(69,43){\Large $\mathcal{B}'$}
\end{overpic}
\caption{$\mathcal{B}'=R(p,-2,m,0)$.}
\label{fig:montesions_C}
\end{figure}

\begin{lemma}\label{lm:slope_estimate_typeK}\hfill
\begin{enumerate}[label=\textnormal{(\roman*)}]
\item\label{itm:p_not_one} If $p\neq 0, 1$, then $\vert r_c\vert\geq \frac{16}{3}$.
\item\label{itm:n_not_one} If $n\neq 0,1$, then $\vert r_c \vert \leq \frac{4}{5}$.
\item\label{itm:pn_one} If $p=1$ or $n=1$, then $\frac{4}{3}\leq r_c\leq \frac{8}{3}$.  
\end{enumerate} 
\end{lemma}
\begin{proof}
First observe that $r_c=4[m,-2,p]$ when $n=0$, and 
$r_c=[-2,1-m,-2,n,0]$ when $p=0$.

To see \ref{itm:p_not_one}, we note that $p\neq 0, 1$ implies $n=0$, so $m\geq 2$ or $m\leq -1$. This, together with $p\geq 2$ or $p\leq -1$, implies  
\[
\frac{16}{3}=4[2,-2,2]\geq 4[m,-2,p] \text{  or  }  4[m,-2,p] \leq 4[-1,-2,-1]=-\frac{16}{3},
\]
and hence the assertion. 
Similarly, for the case \ref{itm:n_not_one}, since $n\neq 0, 1$, we have $n\geq 2$ or $n\leq -1$, and $m\geq 1$ or $m\leq -1$. Therefore
\[
-\frac{4}{5}= [-2,0,-2,-1,0]\geq  [-2,1-m,-2,1,0]\leq [-2,2,-2,2,0] =\frac{4}{5}.
\]  

Consider now the case \ref{itm:pn_one}. If $p=1$, then $n=0$, and hence 
$m\geq 2$ or $m\leq -1$. Therefore 
\[
\frac{4}{3}=4[2,-2,1]\leq  4[m,-2,1]\leq 4[-1,-2,1]=\frac{8}{3}.
\] 
If $n=1$, then $m\leq -2$ or $m\geq 1$, and hence $1-m\geq 3$ or $1-m\leq 0$. Thus   
\[
\frac{4}{3}= [-2,0,-2,1,0]\leq  [-2,1-m,-2,1,0]\leq [-2,3,-2,1,0] =\frac{8}{3}.
\]  
\end{proof}

\begin{theorem}\label{teo:gordon_luecke_K}
Let $\mc V,\mc V'$ be two type $K$ handlebody-knots. Then 
$\mc V\simeq \mc V'$ if and only if 
the slopes of their characteristic annuli are the same.
%they have the same enhanced relative JSJ-graphs.
\end{theorem}
\begin{proof}
``$\Rightarrow$" is clear. 
To see ``$\Leftarrow$", we note that by Theorem \ref{teo:typeK}, $\mc V\simeq \Vl\para$ and $\mc V'\simeq \Vl\paraprime$. %Since 
%they have the same enhanced relative JSJ-graph, 
%$r_c=r_c'$. 
Without loss of generality, it may be assumed one of the following happens: $p=p'=0$, $n= 0,n'=0$ or $p\neq 0,n'\neq 0$.

If $p=p'=0$, then 
$\frac{4m}{4mn-2m+1}=\frac{4m'}{4m'n'-2m'+1}$ 
by Lemma \ref{lm:slope_rc}. In particular, $m=m'$ 
or $-m=m'$; the former implies $(m,n)=(m',n')$, while the latter implies $4mn-2m+1=-(4m'n'-2m'+1)=4mn'-2m-1$, and hence 
$4m(n-n')=-2$, an impossibility. 
Similarly, if $n=n'=0$, then 
$\frac{8pm-4p-4m}{2m-1}=\frac{8p'm'-4p'-4m'}{2m'-1}$.
Comparing the denominator shows that $m=m'$ or $1-m=m'$; the former implies $(m,p)=(m',p')$, while the latter implies
$8pm-4p-4m=-(8p'm'-4p'-4m')=8p'm-4p'-4m+4$, and hence 
$(2m-1)(p-p')=1$, contradicting $2m-1\neq \pm 1$, given $m\neq 0,1$ since $n=0$. Thus in both cases: $p=p'=0$ and $n=n'=0$, $\mc V\simeq \mc V'$.

Suppose now $p\neq 0, n'\neq 0$, then by Lemma \ref{lm:slope_estimate_typeK}, $r_c=r_c'$ can happen only if $p=1$ and $n'=1$. This implies that
\[
\frac{4(m-1)}{2m-1}=r_c=r_c'=\frac{4m'}{2m'+1}.
\] 
Solving the equation for $m,m'$ gives us $m-1=m'$. By Lemma \ref{lm:mirror}, we obtain 
\[\mc V\simeq \Vl(l,m,0,1)\simeq \mirrormc V_L(-l,1-m,0,0)\simeq \Vl(l,m-1,1,0)\simeq \mc V'.\] 
%The assertion thus follows.
\end{proof}

\begin{corollary}\label{cor:homeo_exteriors}
$\Vl(\ast,m,0,p)$ and $\Vl(\ast,m',0,p')$ have homeomorphic exteriors if and only if 
$m=m'$ or $m=1-m'$. 
\end{corollary}
\begin{proof}
They have homeomorphic exteriors if and only if 
the slopes of their characteristic annuli are the same, modulo $\mathbb{Z}$ and up to sign. 
\end{proof}

\begin{proof}[Proof of Corollary \ref{intro:cor:infinite_family}] 
The first assertion follows from Corollary \ref{cor:homeo_exteriors}.
To see that Lee-Lee's second handlebody-knot family $\{L_t\}_{t\in\mathbb{Z}}$ \cite{LeeLee:12} 
can be identified with $\{\Vl(\ast,-1,0,p)\}_{p\in\mathbb{Z}}$,
we recall that $L_t$ is of type $K$, for every $t$, by \cite[Example $5.3.8$]{KodOzaWan:24}, and the slopes of the characteristic annulus is
$-\frac{4}{3}+4t$ by \cite{Wan:24i}\footnote{\cite[Theorem $3.6$]{Wan:24i}, in fact, computes the mirror image of Lee-Lee handlebody-knots}. The assertion then follows from Theorem \ref{teo:gordon_luecke_K}. 
\end{proof} 
%\begin{corollary}\label{cor:infinite_family}
%Given $m\neq 0,1$, the set $\{\Vl(\ast,m,0,p)\}_{p\in\mathbb{Z}}$ 
%is an infinite family of inequivalent handlebody-knots with homeomorphic exteriors.
%\end{corollary} 

%\section*{Acknowledgment}

%%%%%%%%%%%%%%%%%%%%%%%%%%%%%%%%%%%%%%%%%%%%%%%%%%%%%%%%%%%%%%%%%%%%%%%%%%%%%

\end{document}